\documentclass{amsproc}
\usepackage[margin=3.5cm]{geometry}
\usepackage{amsmath}
\usepackage{amssymb}
\usepackage{rotating}
\usepackage{graphicx}
\DeclareGraphicsExtensions{.png,.pdf,.eps}
\usepackage[all]{xy}
\usepackage{tikz}
\usetikzlibrary{arrows}
\usetikzlibrary{positioning}
\usepackage{enumitem}
\usepackage{multirow,tabularx}
\usepackage[pdfstartview=FitH]{hyperref}

\newtheorem{theorem}[equation]{Theorem}
\newtheorem{lemma}[equation]{Lemma}
\newtheorem{corollary}[equation]{Corollary}
\newtheorem{proposition}[equation]{Proposition}

\newtheorem{claim}{Claim}
\theoremstyle{definition}

\newtheorem{definition}[equation]{Definition}

\newtheorem{remark}[equation]{Remark}

\newtheorem{example}[equation]{Example}

\newtheorem*{caution}{Caution}
\numberwithin{equation}{section}

\def\bC{{\mathbb C}}

\def\bP{{\mathbb P}}
\def\bQ{{\mathbb Q}}
\def\bR{{\mathbb R}}
\def\bZ{{\mathbb Z}}
\def\bT{{\mathbb T}}

\def\cC{{\mathcal{C}}}
\def\cO{{\mathcal{O}}}
\def\cN{{\mathcal{N}}}
\def\cH{{\mathcal{H}}}

\def\cR{{\mathcal R}}
\def\cF{{\mathcal F}}

\def\cS{{\mathcal S}}

\def\wt{\widetilde}

\def\ovl{\overline}

\def\Hom{\operatorname{Hom}}
\def\Pic{\operatorname{Pic}}

\def\Spec{\operatorname{Spec}}

\def\deg{\operatorname{deg}}

\def\Nef{{\operatorname{Nef}}}
\def\SAmp{{\operatorname{SAmp}}}

\def\Mov{{\operatorname{Mov}}}

\def\Cl{\operatorname{Cl}}

\def\Sp{\operatorname{Sp}}

\definecolor{zielony}{rgb}{0.5, 0.9, 0.1}
\definecolor{czerwony}{rgb}{0.9, 0.2, 0.1}
\definecolor{niebieski}{rgb}{0.3, 0.1, 0.9}

\begin{document}

\title[Cox rings and symplectic quotient singularities with torus action]
 {Cox rings and symplectic quotient singularities with torus action}

\author[M.~Grab]{Maksymilian Grab}
\address{Instytut Matematyki UW, Banacha 2, 02-097 Warszawa, Poland}
\email{M.Grab@mimuw.edu.pl}



\date{\today}

\thanks{
}


\begin{abstract}
We develop a method of finding a Cox ring of a crepant resolution of a quotient singularity with a torus action and apply it to examples of symplectic quotient singularities in dimension 4. 
In addition we obtain a bound on the degrees of homogeneous generators of the Cox ring in a more general setup, based on the multigraded Castelnuovo-Mumford regularity and Kawamata-Viehweg vanishing.
\end{abstract}

\maketitle

\section{Introduction}

In this article we study the interaction between three themes -- torus actions, crepant resolutions and Cox rings -- in the context of quotients $\bC^4/G$ for certain reducible symplectic representations $G$ of finite groups. We consider the binary tetrahedral group, the symmetry group $S_3$ and the wreath product $\bZ_2\wr S_2$ (isomorphic to the dihedral group of order eight). 

By the classical theorem of Hilbert and Noether, an invariant ring $\bC[x_1,\ldots, x_n]^{G},$ where $G$ is a finite group acting linearly on polynomial ring, is a finitely generated algebra. Quotient singularity is the corresponding singular varieties of the form $\bC^n/G = \Spec \bC[x_1,\ldots,x_n]^G$. The study of crepant resolutions of quotient singularities generalizes the theory of minimal resolutions of du Val singularities to dimensions greater than two. It shows an interplay between geometry and the theory of finite groups in the McKay correspondence~\cite{ReidMcKayCorrespondence}.

The Cox ring of a normal algebraic variety $X$ with a finitely generated Weil divisor class group $\Cl(X)$ is a $\Cl(X)$-graded ring (see~\cite{CoxRings} for a precise construction and a detailed exposition):
\begin{equation*}
\cR(X) = \bigoplus_{D\in\Cl(X)}H^0(X,D).
\end{equation*}
If $\cR(X)$ is a finitely generated $\bC$-algebra it gives a powerful tool to study the geometry of $X$ and its small modifications (i.e. modifications in codimension greater than one). In particular one may recover $X$ as a GIT quotient of $\Spec \cR(X)$ by the action of the Picard (quasi)torus $\bT = \Hom(\Cl(X),\bC^*)$. This is the case in our setting, i.e. if $X$ is a crepant resolution of a four-dimensional symplectic quotient singularity, then $\cR(X)$ is a finitely generated algebra, see~\cite[Theorem~3.2]{AW} and~\cite[Theorem~1.2]{WierzbaWisniewski}.

Varieties with an action of an algebraic torus $(\bC^*)^r$ form an important class of objects in algebraic geometry due to their relations with combinatorics and convex geometry. Well-studied examples include toric varieties~\cite{CLS}, but not only, see e.g.~\cite{Tvarieties}. Of the most importance for us will be the recent methods presented in~\cite{BWW}. They enable us to describe combinatorially local data of the torus action around fixed points and to connect them with the cohomology of line bundles via the Lefschetz-Riemann-Roch theorem.

The quotients $\bC^4/G$ that we consider were subjects of previous studies. In particular it was known that the resolution exists, and the explicit constructions of the resolutions were given, see e.g.~\cite{BellamySchedler}, and~\cite{LehnSorger}. Here however we take a different approach, originated in~\cite{CoxSurf} and~\cite{81resolutions}, and developed in~\cite{SymplCox} and~\cite{Yamagishi}, aiming at generalizations to the study of crepant resolutions of other quotient singularities. The idea is to analyze the geometry of crepant resolutions via the algebraic and combinatorial properties encoded by their Cox ring (the same for all such resolutions of a fixed quotient singularity). In particular one would like to find the presentation of the Cox ring and construct a resolution as a GIT quotient. 

We build up on one of the previous attempts to give a general framework for computing Cox rings of crepant resolutions of quotient singularities, developed in~\cite{CoxSurf},~\cite{81resolutions},~\cite{SymplCox},~\cite{Yamagishi}. In fact the (non-minimal) lists of generators of the Cox ring of the crepant resolution of $\bC^4/G$ were found by Yamagishi by an algorithm presented in~\cite[\S 4]{Yamagishi}. In this paper we present an alternative, more geometric method to deal with main difficulties in computing the Cox ring, using a torus action on the resolution. Our motivation is twofold. For once we are eager to test and develop the new methods related to torus actions in the nontrivial context of quotient singularities with an action of the higher complexity. Apart from that we hope that such a method would broaden the applications of the approach to study crepant resolutions via Cox rings, which currently is limited by the computational complexity of the known algorithms.

The starting point of the whole approach is an embedding of the Cox ring into the ring of Laurent polynomials over the invariant ring of the commutator subgroup of group $G$ used to construct the quotient. Using the McKay correspondence one constructs a finitely generated subring $\cR$ which is a candidate to be the whole Cox ring of a resolution. The nontrivial step is to check if this ring is the actual Cox ring.
The method that we propose in this article begins with the study of geometry of a GIT quotient $X = \Spec \cR /\!/\bT$ for appropriately chosen linearization. Combinatorial data given by local properties of a torus action on this quotient helps us to find an invariant open cover by copies of four-dimensional affine space. As a consequence this quotient is a crepant resolution. To prove that ring $\cR$ is actually the whole Cox ring $\cR(X)$ we use the Lefschetz-Riemann-Roch formula~\cite{BFQ} for the equivariant Euler characteristic and the Kawamata-Viehweg vanishing to calculate the generating function for dimensions of weight spaces of movable linear systems on a resolution. The final argument is made possible by reduction to the checking equality of the finite number of graded pieces of both rings $\cR$ and $\cR(X)$. To make such reduction we prove a bound on the degrees of generators of the Cox ring, which works in more general setup than crepant resolutions of the quotient singularities. The main tools here are multigraded Castelnuovo-Mumford regularity and Kawamata-Viehweg vanishing.

The structure of the paper is as follows. Section~\ref{section:CM-estimation} contains the general results on the bounding degrees of generators of the Cox ring. Section~\ref{section:sympl-strategy} gives an outline of our method to obtain the generators of the Cox ring of crepant resolutions for a quotient singularity with a torus action. In Section~\ref{section:LS} we give a detailed application of this strategy in case of a symplectic action of binary tetrahedral group on $\bC^4$. Finally, Section~\ref{section:two-examples} lists analogous results in case of two simpler examples: actions of symmetric group $S_3$ and wreath product $\bZ_2 \wr S_2$.

\subsection*{Acknowledgments}
The author would like to thank Jaros\l{}aw Wi\'sniewski for proposing the research topic and for many helpful discussions. He also would like to express gratitude toward Maria Donten-Bury for the help with the computation of the central fibre and the stable locus.

The author was supported by the Polish National Science Center project of number 2015/17/N/ST1/02329.

\section{Bounding the degrees of generators of the Cox ring}
\label{section:CM-estimation}

In this section we leave for a moment the context of quotient singularities to give a direct bound on the degrees of generators of Cox rings. It will be applied to crepant resolutions of quotient singularities in later part of the article (see sections~\ref{section:sympl-strategy} and~\ref{section:Cox-LS}). To obtain such a bound we use the Kawamata-Viehweg vanishing and properties of the multigraded Castelnuovo-Mumford regularity (as introduced by Maclagan and Smith in~\cite{MaclaganSmith}).

We employ these concepts via the following proposition.

\begin{proposition}\label{proposition:multigraded-regularity}
Assume that $X$ is a smooth algebraic variety with $K_X = 0$. Let $\varphi\colon X\to Y$ be a projective morphism onto an affine variety $Y$. Let $B_1,\ldots, B_\ell$ be globally generated line bundles on $X$. Assume that the fibres of $\varphi$ are of dimension at most $r$. Let $A$ be a line bundle on $X$ such that $A\otimes B^{-\mathbf{u}}$ is $\varphi$-nef and $\varphi$-big for every tuple $\mathbf{u}$ of nonnegative integers with $|\mathbf{u}| \le r$.
Then the multiplication map:
\begin{equation*}
H^{0}(X,A\otimes B^{\mathbf{v}})\otimes H^{0}(X,B^{\mathbf{w}}) \to H^{0}(X,A\otimes B^{\mathbf{v}+\mathbf{w}})
\end{equation*}
is surjective for all $\ell$-tuples of nonnegative integers $\mathbf{v},\mathbf{w}$.
\end{proposition} 

\begin{proof}
This is a direct consequence of~\cite[Thm~2.1(2)]{multigraded_regularity} for $L = \cO_X$ and $\cF = A$ since the defining property of $\cO_X$-regular sheaf follows then by Kawamata-Viehweg vanishing theorem~\cite[Thm~1-2-3]{KMM87}. The vanishing of cohomologies $H^i(X,A\otimes B^{-\mathbf{u}})$ with $i > r$ follows from the assumption on the dimension of fibres of $\varphi$ as $Y$ is affine.
\end{proof}

\subsection*{Degrees of generators of the Cox ring}
Let $X$ is a normal variety with finitely generated Cox ring and $H^0(X,\cO_X^*) = \bC^*$. Let $\varphi\colon X\to Y$ be a projective morphism, where $Y$ is a normal affine variety with torsion Weil divisor class group. 

We are motivated by the following well-known result.
\begin{theorem}\label{thm-mds}
The two following conditions are equivalent:
\begin{enumerate}
\item Cox ring $\cR(X)$ of $X$ is finitely generated $\bC$-algebra
\item There exists a finite polyhedral subdivision $\Sigma$ of the cone $\Mov(X)$ of movable divisors on $X$ by cones $(f_i^{-1})_*(\Nef(X_i/Y))$, where $X_i$ are normal, $\bQ$-factorial varieties projective over $Y$ and with finitely generated Weil divisor class groups, and where $f_i\colon X \dasharrow X_i$ are birational maps over $Y$, which are isomorphisms in codimension $1$. Moreover all nef bundles on $X_i$ are semiample.
\end{enumerate} 
\end{theorem}
\begin{proof}
See for example~\cite[Sect~4.3.3]{CoxRings} for the case where $Y = \Spec \bC$ -- arguments given there generalize directly to our setup, since $\Cl(X/Y)\otimes_{\bZ} \bQ = \Pic(X)\otimes_{\bZ}\bQ$ as $\Cl(Y)$ is a torsion group. 
\end{proof}
The condition (2) can be found in the literature as a defining condition of a (relative) Mori Dream Space (cf.~\cite[Def~2.5]{AW}).

We will be applying the results of this section in a setup where $\varphi$ is a crepant resolution of symplectic quotient singularities in dimension~4. Then the finite generation of Cox ring is a consequence of~\cite[Thm~3.2]{AW}. It remains true more generally, for minimal models of any quotient singularities in arbitrary dimension, by~\cite[Cor~1.3.2]{BCHM}.

Let $\Sigma$ be subdivision of $\Mov(X)$ from Theorem~\ref{thm-mds} and let $\sigma\in \Sigma$ be a cone of maximal dimension. Then $\sigma = \SAmp(X')$ for some $\varphi'\colon X'\to Y$ isomorphic to $X$ over $Y$ in codimension one. Let $\rho_1,\ldots, \rho_s$ be rays of $\sigma$ and let $D_1,\ldots, D_s\in \Cl(X')$ be ray generators. Assume that $m_1,\ldots, m_s >0$ are such that $m_1D_1,\ldots,m_sD_s$ are divisors with base point free linear systems. Let $\cR(X)_{\sigma} = \bigoplus_{\alpha\in \Cl(X)\cap \sigma}\cR(X)_{\alpha}$.

\begin{proposition}\label{prop:CMestimate}
Assume that $K_X = 0$, $X'$ is smooth and the dimensions of fibres of $\varphi'$ are at most $r$. Then the $\bC$-algebra $\cR(X)_{\sigma}$ is generated by the elements corresponding to global sections of all line bundles of the form $\sum_{i=1}^{s}k_iD_i$ with $0\le k_i < (r+1)m_i$ for every $i$.
\end{proposition}
\begin{proof}
We apply Proposition~\ref{proposition:multigraded-regularity} multiple times for line bundles $\{B_1,\ldots, B_{\ell}\} = \{m_iD_i\colon i\in I\}$ and $A =\sum_{i\in I}(rm_i+k_i)D_i$, for all $I\subset \{1,\ldots,s\}$ and all $0\le k_i < m_i$.
\end{proof}

\begin{proposition}\label{prop:Mov-and-fixed-generate-Cox}
Under the assumptions above the Cox ring $\cR(X)$ is generated by generators of $\cR(X)_{\sigma}$ for all $\sigma\in \Sigma$ together with the elements corresponding to the components of the exceptional divisor.
\end{proposition}
\begin{proof}
The argument is similar to the proof of the fact that Mori Dream Spaces have finitely generated Cox ring, i.e. implication~$(2)\implies (1)$, in Theorem~2.2. Cf.~\cite[Thm~4.3.3.1]{CoxRings}.

Every effective Weil divisor on $X$ is of the form $\ovl{D}+E$ where $\ovl{D}$ is the strict transform of a Weil divisor on $Y$ and $E$ is an effective divisor supported on the components of the exceptional divisor. Moreover $\ovl{D}$ is movable as $D$ is movable on $Y$ (because $Y$ is affine). 

Indeed, assume that $D\sim D'$ on $Y$ for an effective divisor $D'$ and $D'$ does not contain a certain prime component $D_0$ of $D$ in its support. Then $\ovl{D}+F_1\sim \ovl{D'}+F_{2}$, where $F_1,F_2$ are effective divisors supported on the components of the exceptional divisor. Let $A$ be a relatively ample divisor. Let $A' = A - \varphi^*\varphi_*A$. Then $A'$ is relatively ample and $-A'$ is an effective divisor supported on the exceptional divisor. By replacing $A'$ with some positive multiple we may assume that $-(A'+F_1)$ is effective and $A'+F_2$ is relatively ample. Let $A''\sim A'+F_2$ be an effective divisor which does not contain $D_0$ in its support. Now $\ovl{D} \sim \ovl{D'}+F_2 - F_1 = \ovl{D'} - A' - F_1 + A + F_2\sim \ovl{D'} - A' - F_1 + A''$ and the latter is an effective divisor which does not contain $\ovl{D_0}$ in its support. As we may argue in the same manner for each prime component $D_0$, we conclude that $\ovl{D}$ is movable.

As $H^0(X,\cO_X^*) = \bC^*$ homogeneous elements of $\cR(X)$ up to multiplication by a constant correspond to effective divisors on $X$. Therefore we see that every homogeneous element of $\cR(X)$ belongs to the $\bC$-algebra generated by all algebras $\cR(X)_{\sigma}$ and the elements corresponding to the components of exceptional divisor.
\end{proof}

Note that Propositions~\ref{prop:CMestimate} and~\ref{prop:Mov-and-fixed-generate-Cox} combined allow one to bound degrees of generators of $\cR(X)$ under the assumption that all the codimension two modifications $X'\to Y$ of $X\to Y$ corresponding to the chambers in the movable cone are smooth. This assumption is satisfied for example in the case of three-dimensional quotient singularities and in the case of symplectic quotient singularities considered in this paper. 

\section{Cox ring of a resolution via torus action -- outline of strategy}
\label{section:sympl-strategy}

The following well-known observation describes the precise relation between birational geometry of crepant resolutions of the symplectic quotient singularity and their Cox ring.
\begin{proposition}\label{prop:4-dim-sympl-Mov-R}
Assume that $G\subset \Sp_{2n}(G)$ is a finite group and there exists a crepant resolution $\varphi\colon X\to \bC^{2n}/G$. Let $f_1,\ldots,f_m$ be the set of generators of the Cox ring $\cR(X)$ and let $d_i = \deg f_i\in \Pic(X)$. Then $\Mov(X) = \bigcap_{i=1}^{m} \sigma_i$, where $\sigma_i = cone(d_{1},\ldots, d_{i-1},d_{i+1},\ldots,d_m)\subset \Pic(X)\otimes_{\bZ} \bQ$. Moreover there is a one-to-one correspondence between the crepant resolutions of $\bC^{2n}/G$ and GIT chambers of the action of the Picard torus $\bT = \Hom(\Pic(X), \bC^*)$ on $\Spec \cR(X)$ contained in $\Mov(X)$. Namely, taking GIT quotients corresponding to GIT chambers of this action we obtain all pairwise nonisomorphic symplectic resolutions of $\bC^{2n}/G$.
\end{proposition}
\begin{proof}
The description of the movable cone in terms of generators of the Cox ring is well-known as well as the correspondence between GIT chambers in $\Mov(X)$ and $\bQ$-factorial birational models of $X$ isomorphic to $X$ in codimension $1$, see e.g.~\cite[Prop~3.2.3.3]{CoxRings} and~\cite[Thm~3.1.4.3~and~Sect~1.6.3]{CoxRings} respectively. 

It remains to show that every such birational model is smooth, so it is a resolution. This follows by~\cite[Cor~31]{NamikawaFlops} or by a more general result~\cite[Cor~25]{NamikawaFlops}. 
\end{proof}

From now on we will assume that $G\subset\Sp_4(\bC)$ is finite group such that $\bC^{2n}/G$ admits crepant resolution, i.e. we are in the situation from the Proposition~\ref{prop:4-dim-sympl-Mov-R} with $n=2$. Moreover we additionally assume that $\bC^4 = V\oplus V'$ as $G$-representations, with $\dim V = 2 = \dim V'$. By choosing appropriate basis we have then $\bC^4 = \bC^2\oplus \bC^2$, where $G$ acts linearly on each copy of $\bC^2$.

By taking $\bC^*$-action given by multiplication by scalars on each $\bC^2$ we obtain a two-dimensional algebraic torus action on the singularity, which lifts to the crepant resolution~\cite[Thm~1.3]{KaledinSelecta}. We are interested in to finding the Cox ring and study geometry of these resolutions. In this section we formulate briefly the general strategy which we illustrate in the remaining part of the paper.

\begin{caution}
To avoid confusion we once more emphasize the fact that in this and following sections we consider simultaneously two different tori. One is the Picard torus $\bT= \Hom(\Pic(X),\bC^*)$ of a resolution $X\to \bC^4/G$ and the other one is the torus $T = (\bC^*)^2$ acting on $\bC^4= \bC^2\oplus \bC^2$ by multiplication by scalars on each component $\bC^2$. Torus $T$ acts also on the quotient and (equivariantly) on resolutions.
\end{caution}

First, using the general method of finding generators for the Cox ring from paper~\cite{SymplCox} we find a `candidate set' for a set of generators of the Cox ring. We form a subalgebra $\cR$ of the Cox ring generated by this set and compute the cone $\Mov(\cR)$ together with its GIT subdivision with respect to the Picard torus action on $\Spec \cR$ induced by the $\Cl(X)$-grading. This step, together with the formulation of the final result, is done in Section~\ref{section:LS-generators}.

Then, we study geometry of the GIT quotient with respect to a linearization $\lambda$ from a chamber of the GIT subdivision of $\Mov(\cR)$. We start with the study of the central fibre, i.e. the fibre over the point $[0]\in \bC^{4}/G$ given by a natural morphism $(\Spec \cR)^{ss}(\lambda)/\!/\bT \to \bC^4/G$. This is done in Section~\ref{section:central-fibre}.

Next, we present an open cover of the GIT quotient consisting of affine spaces, thus proving the smoothness of the quotient. At this point we are also able to prove that the GIT quotient is a crepant resolution. This step is done in Section~\ref{section:smoothness-of-GIT-quotient}.

Then, in Section~\ref{section:compasses} we explain how we found the open cover. Working under the assumption of smoothness of the quotient, we find \emph{compasses} at fixed point (see Definition~\ref{def-compass}) of two-dimensional torus action on the GIT quotient and give the heuristic argument that was used to predict the cover.

To prove that $\cR$ is actually the whole Cox ring, we employ results of Section~\ref{section:CM-estimation} and the Lefschetz-Riemann-Roch theorem~\cite[Appx~A]{BWW}. 
First, we prove that $\Mov(X)$ is described by degrees of chosen generators of $\cR$ as if $\cR$ were the Cox ring (cf. Proposition~\ref{prop:4-dim-sympl-Mov-R}) and that the GIT subdivision of this cone is the same as the subdivision induced by the Picard torus action on $\cR(X)$. We also calculate the Hilbert series for the subring $\cR(X)_{\ge 0}:=\bigoplus_{L\in \Mov(X)}\cR(X)_L$. This is done in Section~\ref{section:hilbert}. Finally, in Section~\ref{section:Cox-LS} we apply the results of Section~\ref{section:CM-estimation} to deduce that $\cR = \cR(X)$.

\section{Binary tetrahedral group}
\label{section:LS}
\subsection{The setup and the final result}\label{section:LS-generators}

Let $G\subset \Sp_{4}(\bC)$ be the symplectic representation of binary tetrahedral group generated by the matrices:
{
\small
\begin{equation*}
\begin{pmatrix}
i & 0 & 0 & 0\\
0 & -i & 0 & 0\\
0 & 0 & i & 0\\
0 & 0 & 0 & -i
\end{pmatrix}, \qquad
-\frac{1}{2}\begin{pmatrix} 
(1+i)\epsilon & (-1+i)\epsilon & 0 & 0 \\
(1+i)\epsilon & (1-i)\epsilon & 0 & 0 \\
0 & 0 & (1+i)\epsilon^2 & (-1+i)\epsilon^2\\
0 & 0 & (1+i)\epsilon^2 & (1-i)\epsilon^2
\end{pmatrix},
\end{equation*}
}
where $\epsilon = e^{2\pi i /3}$ is a third root of unity.

In this section we collect known facts on this group, the corresponding symplectic quotient singularity $\bC^4/G$ and its symplectic resolution $X$, which were investigated previously by Bellamy and Schedler in~\cite{BellamySchedler} and by Lehn and Sorger in~\cite{LehnSorger}. In particular we introduce two $\bZ^2$-gradings on the Cox ring of $X$, one induced by the Picard torus action and one induced by the action of a two-dimensional torus on $X$. This is the first step in the program outlined in the previous section.

Let us recall that the element $g\in G\subset \Sp_{4}(\bC)$ is a symplectic reflection if its fixed-point linear subspace in $\bC^{4}$ has codimension 2. By McKay correspondence of Ito and Reid~\cite{ItoReid} conjugacy classes of such elements are in one-to-one correspondence with crepant divisors over the quotient singularity $\bC^4/G$.

\begin{proposition}\label{LS-structure-of-G}
\leavevmode
\begin{enumerate}
\item There are $7$ conjugacy classes of elements of $G$ among which two consist of symplectic reflections.
\item The commutator subgroup $[G,G]$ has order $8$ and it is isomorphic to the quaternion group. In particular the abelianization $G/[G,G]$ is cyclic of order $3$.
\item The representation $G$ defined above is reducible. It decomposes into two two-dimensional representations $V_1\oplus V_2$. In particular the $(\bC^{*})^{2}$-action on $\bC^{4}$ induced by multiplication by scalars on $V_i$ commutes with $G$.
\end{enumerate}
\end{proposition}
\begin{proof}
Points (1) and (2) can be quickly verified with GAP Computer Algebra System~\cite{GAP4}. Point (3) follows directly by the definition of $G$.
\end{proof}

Let $\Sigma \subset \bC^4/G$ be the singular locus of the quotient. It can be described as follows.
\begin{proposition}\label{LS-sing-quotient}
The preimage of $\Sigma$ via the quotient map $\bC^4\to \bC^4/G$ consists of four planes, each of which maps onto $\Sigma$. Outside the image of $0$ the singular locus is a transversal $A_2$-singularity.
\end{proposition}
\begin{proof}
This is a consequence of a direct computation of the subspaces fixed by symplectic reflections of $G$ -- these are the only elements that stabilize a proper nontrivial subspace of $\bC^4$.
\end{proof}

Let $\varphi\colon X\to \bC^{4}/G$ be a crepant resolution.  It is known that such a resolution exists (see~\cite[Sect~1]{BellamySchedler}) but we will also prove it independently in Section~\ref{section:smoothness-of-GIT-quotient}. Using the symplectic McKay correspondence~\cite{KaledinMcKay} and Propositions~\ref{LS-structure-of-G} and~\ref{LS-sing-quotient} we obtain the following facts about the geometry of $X$.

\begin{proposition}
There are two exceptional divisors $E_1,E_2$ of $X$ each of which is mapped onto $\Sigma$. The central fibre $\varphi^{-1}([0])$ consist of four surfaces. The fibre of $\varphi$ over any point in $\Sigma\setminus [0]$ consists of two curves isomorphic to $\bP^1$ intersecting in one point and each of which is contained in exactly one of two exceptional divisors.
\end{proposition}

By theorem of Kaledin~\cite[Thm~1.3]{KaledinSelecta} and Proposition~\ref{LS-structure-of-G} we have also:

\begin{corollary}\label{cor:T-action-on-resolution}
There is a natural $T:=(\bC^{*})^{2}$-action on $X$ making $\varphi$ an equivariant map.
\end{corollary}

As we noted earlier, we consider two different two-dimensional tori -- one is the Picard torus $\bT = \Hom(\Cl(X), \bC^*)$ torus $T$ acting on $\bC^4= \bC^2\times \bC^2$ by multiplication of scalars on each component $\bC^2$. 

Let $C_i$ be the numerical class of a complete curve which is a generic fibre of the morphism $\varphi|_{E_i}\colon E_i\to \Sigma$. We may describe generators of the Picard group of $X$ in terms of its intersection with curve $C_i$.

\begin{proposition}[cf.{~\cite[2.16]{81resolutions}}]
The Picard group of $X$ is a free rank two abelian group generated by line bundles $L_1,L_2$ such that the intersection matrix $(L_i.C_j)_{i,j}$ is equal to identity matrix.
\end{proposition}

One may see the Cox ring $\cR(X)$ of $X$ as a subring of $\bC[x_1,y_1,x_2,y_2]^{[G,G]}[t_1^{\pm 1}, t_{2}^{\pm 1}]$ (see~\cite[Sect~2.1]{SymplCox} for details). It follows that the action of the two-dimensional torus $T=(\bC^*)^2$ on $X$ induces the action on $\Spec \cR(X)$. The methods given in~\cite[Sect~2.1]{SymplCox} and in~\cite[Sect~5]{3dimCox} suggest to consider the following `candidate set' for generators of $\cR(X)$. It may be verified by the general algorithm of~\cite[Sect~4]{Yamagishi} that this indeed is the set of generators. Denote the following elements of the (Laurent) polynomial ring $\bC[x_{1},y_1,x_2,y_2][t_1^{\pm 1},t_{2}^{\pm 1}]$:

{\tiny

\begin{tabular}{l}
$w_{01} = y_1x_2-x_1y_2,$\\
$w_{02} = x_2^5y_2-x_2y_2^5,$\\
$w_{03} = x_1^5y_1-x_1y_1^5,$\\
$w_{04} = x_1^4+(-4b+2)x_1^2y_1^2+y_1^4,$\\
$w_{05} = x_2^4+(4b-2)x_2^2y_2^2+y_2^4,$\\
$w_{06} = x_1x_2^3+(-2b+1)y_1x_2^2y_2+(-2b+1)x_1x_2y_2^2+y_1y_2^3,$\\
$w_{07} = x_1^3x_2+(2b-1)x_1y_1^2x_2+(2b-1)x_1^2y_1y_2+y_1^3y_2,$\\
\end{tabular}

\bigskip

\begin{tabular}{l}
$w_{11} = (-3bx_1^2x_2^2+(-b+2)y_1^2x_2^2+(-4b+8)x_1y_1x_2y_2+(-b+2)x_1^2y_2^2-3by_1^2y_2^2)t_1,$\\
$w_{12} = (x_2^4+(-4b+2)x_2^2y_2^2+y_2^4)t_1,$\\
$w_{13} = (x_1^3x_2+(-2b+1)x_1y_1^2x_2+(-2b+1)x_1^2y_1y_2+y_1^3y_2)t_1,$\\
$w_{14} = (-5x_1^4y_1x_2+y_1^5x_2-x_1^5y_2+5x_1y_1^4y_2)t_1,$\\
$w_{15} = (x_1y_1x_2^4+2x_1^2x_2^3y_2-2y_1^2x_2y_2^3-x_1y_1y_2^4)t_1,$\\
\end{tabular}

\bigskip

\begin{tabular}{l}
$w_{21} = ((3b-3)x_1^2x_2^2+(b+1)y_1^2x_2^2+(4b+4)x_1y_1x_2y_2+(b+1)x_1^2y_2^2+(3b-3)y_1^2y_2^2)t_2,$\\
$w_{22} = (x_1^4+(4b-2)x_1^2y_1^2+y_1^4)t_2,$\\
$w_{23} = (x_1x_2^3+(2b-1)y_1x_2^2y_2+(2b-1)x_1x_2y_2^2+y_1y_2^3)t_2,$\\
$w_{24} = (y_1x_2^5+5x_1x_2^4y_2-5y_1x_2y_2^4-x_1y_2^5)t_2,$\\
$w_{25} = (-2x_1^3y_1x_2^2-x_1^4x_2y_2+y_1^4x_2y_2+2x_1y_1^3y_2^2)t_2,$\\
\end{tabular}

\bigskip 

\begin{tabular}{l}
$w_{3} = (9x_1^2y_1x_2^3+(-2b+1)y_1^3x_2^3+9x_1^3x_2^2y_2+(6b-3)x_1y_1^2x_2^2y_2+(-6b+3)x_1^2y_1x_2y_2^2-9y_1^3x_2y_2^2+$\\
$+(2b-1)x_1^3y_2^3-9x_1y_1^2y_2^3)t_1t_2$\\
\\
$s = t_1^{-2}t_{2},$\\
$t = t_1t_2^{-2},$\\
\end{tabular}

}
where $b$ is a primitive root of unity of order $6$.
\begin{theorem}\label{thm:Cox-ring}
The Cox ring $\cR(X)$ of $X$ is isomorphic to the algebra generated by 20 generators:
\begin{equation*}
w_{01},\ldots,w_{07},w_{11},\ldots,w_{15},w_{21},\ldots,w_{25},w_{3},s,t.
\end{equation*} 
The degree matrix of these generators with respect to the generators $L_1,L_2$ of $\Pic(X)$ (first two rows) and with respect to the $T$-action (remaining two rows) is:
\setcounter{MaxMatrixCols}{20}
\begin{equation*}
\left(\begin{smallmatrix}
w_{01} & w_{02} & w_{03} & w_{04} & w_{05} & w_{06} & w_{07} & w_{11} & w_{12} & w_{13} & w_{14} & w_{15} & w_{21} & w_{22} & w_{23} & w_{24} & w_{25} & w_{3} & s & t \\
\\
\hline \\
& & & & & & & & & & & & & & & & & & & &\\
0 & 0 & 0 & 0 & 0 & 0 & 0 & 1 & 1 & 1 & 1 & 1 & 0 & 0 & 0 & 0 & 0 & 1 & -2 & 1\\
0 & 0 & 0 & 0 & 0 & 0 & 0 & 0 & 0 & 0 & 0 & 0 & 1 & 1 & 1 & 1 & 1 & 1 & 1 & -2\\
\\
\hline \\
& & & & & & & & & & & & & & & & & & & &\\
1 & 0 & 6 & 4 & 0 & 1 & 3 & 2 & 0 & 3 & 5 & 2 & 2 & 4 & 1 & 1 & 4 & 3 & 0 & 0\\
1 & 6 & 0 & 0 & 4 & 3 & 1 & 2 & 4 & 1 & 1 & 4 & 2 & 0 & 3 & 5 & 2 & 3 & 0 & 0
\end{smallmatrix}\right)
\end{equation*}
\end{theorem}
We prove this theorem in Section~\ref{section:Cox-LS}. From now on we will denote the ring generated by elements from the statement of Theorem~\ref{thm:Cox-ring} by $\cR$.

It is a general principle that using the degrees of generators of the Cox ring of $X$ we can describe the movable cone $\Mov(X)$ and find the number of resolutions and the corresponding subdivision of $\Mov(X)$ into the nef cones of resolutions of $X$ (see Theorem~\ref{thm-mds}). However, since we use the description of movable cone and its subdivision to prove that $\cR$ is a Cox ring of $X$ we give independent proofs.

\begin{proposition}\label{prop-mov-cone}\leavevmode
\begin{enumerate}
\item The cone $\Mov(X)$ of movable divisors of $X$ is equal to the cone generated by the line bundles $L_1$ and $L_2$.
\item There are two symplectic resolutions of $\bC^{4}/G$. The chambers in $\Mov(X)$ corresponding to the nef cones of these resolutions are relative interiors of cones $cone(L_1,L_1+L_2)$ and $cone(L_2,L_1+L_2)$. The Mori cones of corresponding resolutions are $cone(C_2, C_1-C_2)$ and $cone(C_1, C_2-C_1)$.
\end{enumerate}
\end{proposition}
\begin{proof}\leavevmode
\begin{enumerate}
\item This follows from~\cite[Thm~3.5]{AW}.
\item We will prove the first part of the claim in Section~\ref{section:hilbert} as Proposition~\ref{proposition-mov-subdivision} (we will not use this result until then). The part on Mori cones then follows by taking dual cones.
\end{enumerate}
\end{proof}

The next theorem also follows from Theorem~\ref{thm:Cox-ring}, but as we need it before we prove this theorem it is proven independently in the Section~\ref{section:smoothness-of-GIT-quotient}.
\begin{theorem}
Taking GIT quotients of $\Spec \cR$ by the Picard torus action with respect to the linearization given by a character $(a,b)$ with $a>b>0$ and with $b>a>0$ one obtains the two symplectic resolutions of $\bC^{4}/G$.
\end{theorem}

\begin{remark}\label{remark:weights-example}
The weights of the $T$-action on global sections of the fixed line bundle $L$ on $X$ are lattice points in $\bZ^2$. Taking a convex hull one obtains a lattice polyhedron in $\bR^2$. For example fixing a line bundle $L=2L_1+L_2$ one gets a polyhedron with the recession cone equal to the positive quadrant of $\bR^2$ and with a head spanned by the lattice points from the picture below:
\begin{center}
\begin{tikzpicture}[scale=0.3]
\foreach \Point in {(0,16),(0,20),(0,22),(1,11),(1,13),(1,15),(1,17),(1,19),(1,21),(2,10),(2,12),(2,14),(2,16),(2,18),(2,20),(2,22),(3,7),(3,9),(3,11),(3,13),(3,15),(3,17),(3,19),(3,21),(4,8),(4,10),(4,12),(4,14),(4,16),(4,18),(4,20),(4,22),(5,5),(5,7),(5,9),(5,11),(5,13),(5,15),(5,17),(5,19),(5,21),(6,4),(6,6),(6,8),(6,10),(6,12),(6,14),(6,16),(6,18),(6,20),(6,22),(7,5),(7,7),(7,9),(7,11),(7,13),(7,15),(7,17),(7,19),(7,21),(8,4),(8,6),(8,8),(8,10),(8,12),(8,14),(8,16),(8,18),(8,20),(8,22),(9,3),(9,5),(9,7),(9,9),(9,11),(9,13),(9,15),(9,17),(9,19),(9,21),(10,2),(10,4),(10,6),(10,8),(10,10),(10,12),(10,14),(10,16),(10,18),(10,20),(10,22),(11,3),(11,5),(11,7),(11,9),(11,11),(11,13),(11,15),(11,17),(11,19),(11,21),(12,2),(12,4),(12,6),(12,8),(12,10),(12,12),(12,14),(12,16),(12,18),(12,20),(12,22),(13,3),(13,5),(13,7),(13,9),(13,11),(13,13),(13,15),(13,17),(13,19),(13,21),(14,2),(14,4),(14,6),(14,8),(14,10),(14,12),(14,14),(14,16),(14,18),(14,20),(14,22),(15,1),(15,3),(15,5),(15,7),(15,9),(15,11),(15,13),(15,15),(15,17),(15,19),(15,21),(16,2),(16,4),(16,6),(16,8),(16,10),(16,12),(16,14),(16,16),(16,18),(16,20),(16,22),(17,1),(17,3),(17,5),(17,7),(17,9),(17,11),(17,13),(17,15),(17,17),(17,19),(17,21),(18,2),(18,4),(18,6),(18,8),(18,10),(18,12),(18,14),(18,16),(18,18),(18,20),(18,22),(19,1),(19,3),(19,5),(19,7),(19,9),(19,11),(19,13),(19,15),(19,17),(19,19),(19,21),(20,0),(20,2),(20,4),(20,6),(20,8),(20,10),(20,12),(20,14),(20,16),(20,18),(20,20),(20,22),(21,1),(21,3),(21,5),(21,7),(21,9),(21,11),(21,13),(21,15),(21,17),(21,19),(21,21),(22,2),(22,4),(22,6),(22,8),(22,10),(22,12),(22,14),(22,16),(22,18),(22,20),(22,22),(23,1),(23,3),(23,5),(23,7),(23,9),(23,11),(23,13),(23,15),(23,17),(23,19),(23,21),(24,0),(24,2),(24,4),(24,6),(24,8),(24,10),(24,12),(24,14),(24,16),(24,18),(24,20),(24,22),(25,1),(25,3),(25,5),(25,7),(25,9),(25,11),(25,13),(25,15),(25,17),(25,19),(25,21),(26,0),(26,2),(26,4),(26,6),(26,8),(26,10),(26,12),(26,14),(26,16),(26,18),(26,20),(26,22)}{
\draw[black] \Point circle[radius=5pt];
\fill[black] \Point circle[radius=5pt];}

\draw[black] (0,22)--(0,20)--(0,16)--(1,11)--(3,7)--(5,5)--(6,4)--(10,2)--(20,0)--(24,0)--(26,0);

\coordinate (A) at (0,16);
\node at (A) [left = 1mm of A]{\small $(0,16)$};
\coordinate (B) at (1,11);
\node at (B) [left = 1mm of B]{\small $(1,11)$};
\coordinate (C) at (3,7);
\node at (C) [left = 1mm of C]{\small $(3,7)$};
\coordinate (D) at (5,5);
\node at (D) [left = 1mm of D]{\small $(5,5)$};
\coordinate (E) at (6,4);
\node at (E) [left = 1mm of E]{\small $(6,4)$};
\coordinate (F) at (10,2);
\node at (F) [below left = 1mm of F]{\small $(10,2)$};
\coordinate (G) at (20,0);
\node at (G) [below = 1mm of G]{\small $(20,0)$};
\end{tikzpicture}
\end{center}
By~\cite[Lem~2.4(c)]{BWW} if $L$ is globally generated, then marked vertices of this polyhedron correspond to $T$-fixed points of $X$ where $X$ is the resolution on which $L$ is relatively ample. We will see in~\ref{lemma-polytopes} that indeed fixed points of this polytope correspond to points in $X^T$ and in Lemma~\ref{lemma:ggbundles} that $L$ is globally generated.
\end{remark}

\subsection{The structure of the central fibre}\label{section:central-fibre}

In this section we study the structure of the central fibre $\varphi^{-1}([0])$ of such a resolution $\varphi\colon X\to \bC^4/G$ using the ideal of relations between generators of the ring $\cR$, under the assumption that $X = \Spec \cR/\!/_{L}\bT$ for some linearization $L$.The results of this section are used in the next one, where we investigate the action of the two-dimensional torus $T$ on $X$ with the fixed point locus $X^T$ contained in the central fibre. The additional assumption that $X = \Spec \cR/\!/_L\bT$ is dealt with in Section~\ref{section:smoothness-of-GIT-quotient}.
\begin{lemma}\label{lemma:inclusion-of-invariants}
We have an isomorphism $\Spec \cR^{\bT}\cong \bC^{4}/G$. In particular the inclusion of invariants $\cR^{\bT}\subset \cR$ induce map $p\colon \Spec \cR \to \bC^4/G$.
\end{lemma}
\begin{proof}
This follows from the fact that $\cR^{\bT} = \bC[x_1,x_2,x_3,x_4]^{G}$ via the embedding $\cR\subset \bC[x_1,x_2,x_3,x_4]^{[G,G]}[t_1^{\pm 1},t_2^{\pm 1}]$.
\end{proof}

Let $Z=p^{-1}([0])$. Decomposing the ideal of relations from the presentation of $\Spec \cR$ one obtains the decomposition of $Z$ into irreducible components. We consider the closed embedding $\Spec \cR\subset \bC^{20}$ given by the generators of $\cR$ from statement of Theorem~\ref{thm:Cox-ring}.

\begin{proposition}\label{prop:compontents-over-fibre}
The components of $Z$ are the following subvarieties of $\bC^{20}$:
{\tiny
\begin{align*}
Z_u & = V(w_{3},w_{ij}\ |\ (i,j)\in {(0,1),\ldots,(0,7),(1,1),\ldots,(1,5),(2,1),\ldots,(2,5)}),\\
Z_0 & = V(s, t, w_{25}, w_{24}, w_{15}, w_{14}, w_{07}, w_{06}, w_{05}, w_{04}, w_{03}, w_{02}, w_{01},\\
& w_{12}w_{22}-w_{13}w_{23},\ w_{11}w_{21}-9w_{13}w_{23},\ w_{21}^3-27w_{22}w_{23}^2,\ w_{13}w_{21}^2-3w_{11}w_{22}w_{23},\\
& w_{12}w_{21}^2-3w_{11}w_{23}^2,\ 3w_{13}^2 w_{21}-w_{11}^2 w_{22},\ 3 w_{12}w_{13}w_{21}-w_{11}^2 w_{23},\ w_{11}^3-27w_{12}w_{13}^2)\\
Z_P & = V(s, w_{25}, w_{24}, w_{23}, w_{22}, w_{21}, w_{15}, w_{14}, w_{07}, w_{06}, w_{05}, w_{04}, w_{03}, w_{02}, w_{01},\\
& w_{11}^3-27w_{12}w_{13}^2-i\sqrt{3}w_{3}^2 t),\\
Z_P' & = V(t, w_{25}, w_{24}, w_{15}, w_{14}, w_{13}, w_{12}, w_{11}, w_{07}, w_{06}, w_{05}, w_{04}, w_{03}, w_{02}, w_{01},\\
& w_{21}^3-27w_{22}w_{23}^2+i\sqrt{3}w_{3}^2 s),\\
Z_1 & = V(s, w_{25}, w_{24}, w_{23}, w_{21}, w_{15}, w_{12}, w_{07}, w_{06}, w_{05}, w_{04}, w_{03}, w_{02}, w_{01},\\
& 2w_{3}w_{22}t+\zeta_3w_{11}w_{14},\ 2w_{11}^2w_{22}+\zeta_{12}^{7}\sqrt{3}w_{3}w_{14},\ w_{11}^3-i\sqrt{3}w_{3}^2 t,\ 4w_{11}w_{22}^2 t+\zeta_{12}^{5}\sqrt{3}w_{14}^2),\\
Z_2 & = V(t, w_{25}, w_{22}, w_{15}, w_{14}, w_{13}, w_{11}, w_{07}, w_{06}, w_{05}, w_{04}, w_{03}, w_{02}, w_{01},\\
& 2w_{3}w_{12}s+\zeta_6 w_{21}w_{24},\ w_{21}^3+i\sqrt{3}w_{3}^2 s,\ 2w_{12}w_{21}^2+\zeta_{12}^{11}\sqrt{3}w_{3}w_{24},\ 4w_{12}^2 w_{21}s+\zeta_{12}^{7}\sqrt{3}w_{24}^2),
\end{align*}
}

where $\zeta_3, \zeta_6, \zeta_{12}$ are primitive 3rd, 6th and 12th roots of unity. The component $Z_{u}$ is contained in the locus of unstable points with respect to any linearization of the Picard torus via character from the movable cone. Points in the component $Z_P'$ are unstable with respect to any linearization by a character $(2,1)$ and points in the component $Z_P$ are unstable with respect to any linearization by a character $(1,2)$.
\end{proposition}
\begin{proof}
The first part follows by decomposing ideal of $Z$ in Singular. The statements concerning stability are the consequences of Lemma~\ref{lemma:stability} below.
\end{proof}

\begin{lemma}\label{lemma:stability}
The unstable locus of $\Spec \cR$ with respect to a linearization of the trivial line bundle by a character $(2,1)$ is cut out by equations:
\begin{equation*}
w_{12}s = w_{12}w_{23}=w_{11}w_{3} = w_{12}w_{3} = w_{13}w_{3} = w_{13}w_{22} = w_{22}t = 0.
\end{equation*}
Moreover, all the semistable points of $Z$ are stable and have trivial isotropy groups.
\end{lemma}
\begin{proof}
This can be checked with a computer calculation, using the Singular library gitcomp by Maria Donten-Bury (see~\url{www.mimuw.edu.pl/~marysia/gitcomp.lib}).
\end{proof}

The following theorem gives a description of components of the central fibre. Let $W$ be the locus of stable points of $\Spec \cR$ with respect to the $\bT$-action linearized by a character $(2,1)$ (the case $(1,2)$ is analogous) and consider the quotient map $W\to X$. Denote by $S_0,S_1,S_2,P$ the images of sets of stable points of the components $Z_0,Z_1,Z_2,Z_P$. Note that these are precisely the components of the central fibre of $X$.

\begin{theorem}\label{thm:central-fibre}\leavevmode
\begin{enumerate}[label=(\alph*)]
\item $S_0$ is a non-normal toric surface whose normalization is isomorphic to the Hirzebruch surface $\cH_{6}$. The action of $T$ on the normalization of $S_0$ is given by characters in the columns of the matrix $\left(\begin{smallmatrix} 1 & -1\\ -1 & -1\end{smallmatrix}\right)$.
\item $S_1$ is a non-normal toric surface whose normalization is the toric surface of a fan spanned by rays:
$(0,1),(1,0),(1,-1),(-1,-2)$. The action of $T$ on the normalization of $S_1$ is given by characters in the columns of the matrix $\left(\begin{smallmatrix} 3 & -1\\ 1 & -1\end{smallmatrix}\right)$.
\item $S_2$ is a non-normal toric surface whose normalization is the toric surface of a fan spanned by rays:
$(0,1),(1,-1),(-1,-2)$. The action of $T$ on the normalization of $S_2$ is given by characters in the columns of the matrix $\left(\begin{smallmatrix} -1 & 3\\ -1 & 1\end{smallmatrix}\right)$.
\item $P$ is isomorphic to $\bP^{2}$. The action of $T$ on $P$ in homogeneous coordinates is given by the matrix $\left(\begin{smallmatrix} 2 & 0 & 3\\ 2 & 4 & 1\end{smallmatrix}\right)$.
\end{enumerate}
\end{theorem}

Proofs of (a)--(c) follow a fairly standard procedure of division of a toric variety by a torus action. We outline an argument in the case of $S_0$, and then we comment on (d).

\begin{proof}[Proof of (a)]
\begin{claim}
By rescaling variables we may assume that $Z_0$ is the toric variety embedded into $\bC^7$ with coordinates $w_{11},w_{12},w_{13},w_{21},w_{22},w_{23},w_{3}$ defined by the toric ideal generated by binomials:
\begin{align*}
& w_{12}w_{22}-w_{13}w_{23},\ w_{11}w_{21}-w_{13}w_{23},\ w_{21}^3-w_{22}w_{23}^2,\ w_{13}w_{21}^2-w_{11}w_{22}w_{23},\\
& w_{12}w_{21}^2-w_{11}w_{23}^2,\ w_{13}^2 w_{21}-w_{11}^2 w_{22},\ w_{12}w_{13}w_{21}-w_{11}^2 w_{23},\ w_{11}^3-w_{12}w_{13}^2.
\end{align*}
\end{claim}
\begin{proof}[Proof of the claim]
This is a general argument (over an algebraically closed field) to reduce a prime binomial ideal that does not contain monomials to a toric ideal: one takes a point $(a_{ij},a_{3})$ with all $a_{ij}\neq 0\neq a_{3}$ in the zero set of the original ideal and then set new coordinates $w_{ij}' = \frac{1}{a_{ij}}w_{ij}, \ w_{3}' = \frac{1}{a_{3}}w_{3}$.
\end{proof}

\begin{claim}
The normalization of $Z_0$ is the affine variety of the cone $\sigma^{\vee} = cone(-2e_2+3e_3,e_1,e_1+3e_2-3e_3,e_2,e_4)$ in space $\bR^4$ spanned by the character lattice $M=\bZ^4$ of four-dimensional torus.
\end{claim}
\begin{proof}[Proof of the claim]
By a calculation as in proof of~\cite[Thm~1.1.17]{CLS} $Z_0$ is the affine toric variety defined by lattice points in $M = \bZ^4$:
\begin{align*}
& v_{11} = e_1+e_2-e_3,\quad v_{12} = e_1+3e_2-3e_3,\quad v_{13} = e_1,\\
& v_{21} = e_3,\quad v_{22} = -2e_2+3e_3,\quad  v_{23} = e_2, \quad v_3 = e_4.
\end{align*}
In other words there is a short exact sequence $0\to L \to \bZ^7\stackrel{\varphi}{\to} M = \bZ^4\to 0$,
where $\varphi$ sends the canonical base $e_i$ to the vectors $v_{ij}$ and $v_3$ in order written as above and $L$ is a subgroup of $\bZ^7$ defined by the binomial ideal of $Z_0$, i.e. $L$ is generated by:
\begin{align*}
& (0,1,-1,0,1,-1,0),(-1,1,0,2,0,-2,0),(1,0,-1,1,0,-1,0),(-2,0,2,1,-1,0,0),\\
& (0,0,0,3,-1,-2,0),(-2,1,1,1,0,-1,0),(-1,0,1,2,-1,-1,0),(3,-1,-2,0,0,0,0).
\end{align*}
One checks that the vectors $v_{11}$ and $v_{21}$ are not needed to generate $\sigma^{\vee}$.
\end{proof}

\begin{claim}
The Picard torus acting on $Z_0$ may be viewed as $(\bC^{*})^2$ embedded into $T_M = (\bC^{*})^4$ by characters in the columns of the matrix:
\begin{equation*}
\begin{pmatrix}
1 & 0 & 0 & 1\\
0 & 1 & 1 & 1
\end{pmatrix}.
\end{equation*}
\end{claim}
\begin{proof}[Proof of the claim]
The $i$-th column corresponds to $e_i\in M,\ i =1,2,3,4$ and these correspond to $w_{13},w_{23},w_{21},w_{3}$ respectively.
\end{proof}

To obtain $\wt{S}=S_0$ from $Z_0$ we have to remove unstable orbits of $Z_0$ and divide the remaining open subset by the action of the Picard torus. By Lemma~\ref{lemma:stability} we have:
\begin{claim}
The unstable locus of $Z_0$ is the union of the closures of two orbits:
\begin{equation*}
O_1 = \{w_{11} = w_{12} = w_{13} = 0\},\qquad O_2=\{w_{21}=w_{22}=w_{23}=w_{3} = 0\}.
\end{equation*}
\end{claim}

One checks by~\cite[3.2.7]{CLS} that:
\begin{claim}
The orbits $O_1,\ O_2$ correspond respectively to the following faces of the cone $\sigma = cone(e_4,e_2+e_3,3e_2+2e_3,e_1,3e_1+e_3)$ dual to $\sigma^{\vee}$:
\begin{equation*}
\tau_1 = cone(e_1),\qquad \tau_2 = cone(e_4,e_2+e_3).
\end{equation*}
\end{claim}

To obtain the fan of the toric variety $S=S_0$ we consider the fan of $Z_0$, remove cones $\tau_1$ and $\tau_2$ together with all the cones containing them and take the family of the images of the remaining cones via the dual of the kernel of the matrix from Claim~3, i.e. by:
\begin{equation*}
Q := \begin{pmatrix}
0 & -1 & 1 & 0\\
-1 & -1 & 0 & 1
\end{pmatrix}.
\end{equation*}

Now $Q(e_4) = (0,1), \ Q(e_2+e_3) = (0,-1), \ Q(3e_2 + 2e_3) = (-1,-3), \ Q(3e_1+e_3) = (1,-3)$ and changing coordinates in $\bZ^2$ by $(1,0)\mapsto (1,-3),\ (0,1)\mapsto (0,-1)$ we obtain the standard fan for Hirzebruch surface $\cH_{6}$.

The matrix of the $T$-action on $S_0$ is given by the product of the matrix $J$ of embedding of $T$ into the big torus of $Z_0$ with the transpose of the matrix $Q$. Here, inspecting the construction of the generators of $\cR$, we obtain:
\begin{equation*}
J = \begin{pmatrix}
3 & 1 & 2 & 3\\
1 & 3 & 2 & 3
\end{pmatrix}.
\end{equation*}
\end{proof}
\begin{proof}[Proof of (d)]
Denote $x:= w_{11},\,y:= w_{12},\,z:= w_{13},\,w:= w_{3}$. Rescaling coordinates we may assume that $Z_{P}$ is embedded in $\bC^{5}$ as the hypersurface $x^3 - yz^2 +w^2t = 0$. Then the unstable locus is described by equations $xw = yw = zw = 0$, see Lemma~\ref{lemma:stability}. This gives an open cover of the set of semistable points, given by the union of three open sets $\{xw\neq 0\}\cup \{yw\neq 0\}\cup \{zw\neq 0\}$. Gluing the quotients of these open sets coincides with the standard construction of $\bP^2$ (with coordinates $x,y,z$) by gluing three affine planes. Finally, note that $T$ acts on $w_{11},w_{12},w_{13}$ with weights $(2,2), (0,4),(3,1)$ respectively.
\end{proof}

We can also describe the non-normal locus of components of the central fibre.

\begin{theorem}\label{thm:non-normal-locus} If $\nu_i:\wt{S}_i\to S_i$ is the normalization of the component $S_i$ of the central fibre $i=0,1,2$ and $\cN_i\subset S_i$ is the locus of the non-normal points of $S_i$, then
\begin{enumerate}[label=(\alph*)]
\item $\nu_0^{-1}(\cN_0)$ is the sum of the closures of the orbits corresponding to $(-1,-3)$ and $(1,-3)$ i.e. it is the sum of invariant fibres of $\cH_6$.
\item $\nu_1^{-1}(\cN_1)$ is the sum of the closures of the orbits corresponding to $(-1,-2)$ and $(1,-1)$.
\item $\nu_2^{-1}(\cN_2)$ is the sum of the closures of the orbits corresponding to $(-1,-2)$ and $(1,-1)$.
\end{enumerate}
\end{theorem}
\begin{proof}
We use Lemma~\ref{lemma:non-normal-points-toric-variety} and the description of $\bT$-stable orbits of $Z_i$ analogous to the one in the proof of Theorem~\ref{thm:central-fibre} for $Z_0$. Altogether, $\bT$-stable orbits of $Z_i$ which consist of normal points correspond to the cones which do not contain the cones $\tau_j$ and are not contained in $\omega_k$, where $\tau_j$'s are defined analogously as in the proof of~\ref{thm:central-fibre}(a) and $\omega_k$'s are:
\begin{itemize}
\item For $i = 0$: $\omega_1 = cone(e_4,e_2+e_3)$ and $\omega_2 = cone(e_1,e_4)$.
\item For $i = 1,2$: $\omega:= \omega_1 = cone(e_2+e_3,e_4)$.
\end{itemize}
In each case one easily finds all such cones and their images via the map $Q$ (again, notation after proof of~\ref{thm:central-fibre}) turn out to be the cones parametrizing orbits in the statement. We conclude since the non-normal points of $S_i = Z_i/\!/\bT$ are precisely the images of the non-normal points of $Z_i$ which are $\bT$-stable. Here we use the fact that all semistable points of $Z_i$ are stable and the isotropy groups of the $\bT$-action are trivial by Lemma~\ref{lemma:stability}, so that the quotient $Z_i^s\to Z_i/\!/\bT$ is a torsor. 
\end{proof}
\begin{lemma}\label{lemma:non-normal-points-toric-variety}
Let $S$ be a subsemigroup of the lattice $M$ and let $U = \Spec \bC[S]$ be the corresponding (not necessarily normal) affine toric variety. Let $\sigma = cone(S)$. Let $\wt{U} = \Spec \bC[M\cap \sigma]$ be the toric normalization of $U$. Then for each subcone $\tau\prec \sigma$ the orbit $O(\tau)\subset U$ consists of the normal points of $U$ if and only if we have the equality of semigroups $M\cap \langle\tau\rangle+ S = M\cap \langle\tau \rangle + M\cap \sigma$, where $\langle \tau\rangle$ is the linear span of $\tau$.
\end{lemma}
\begin{proof}
The toric variety $U(\tau) = \Spec \bC[M\cap \langle\tau\rangle+S]$ is the open subvariety of $\Spec \bC[S]$ obtained by removing these torus orbits which does not contain $O(\tau)$ in their closure. Similarly $\wt{U}(\tau) = \Spec \bC[M\cap \langle\tau \rangle + M\cap \sigma]$ is the open subvariety of $\Spec \bC[M\cap \sigma]$ obtained by removing all the torus orbits whose closure does not contain the orbit corresponding to $\tau$. Moreover $\wt{U}(\tau)\to U(\tau)$ is the normalization. Hence the equality $M\cap \langle\tau\rangle+ S = M\cap \langle\tau \rangle + M\cap \sigma$ holds if and only if $U(\tau)$ is normal, which is precisely the case when $O(\tau)$ consists of the normal points, since it is contained in the closure of each orbit of $U(\tau)$ and the non-normal locus of a variety is closed.
\end{proof}

By the analysis of intersections of stable loci of $Z_i$ and $Z_P$ we can systematically describe the intersections of the components of the central fibre in terms of the identifications from Theorem~\ref{thm:central-fibre}.

\begin{theorem}\label{thm:incidence}
\leavevmode
\begin{enumerate}[label=(\alph*)]
\item $S_0\cap S_1$ is the curve corresponding to $(-1,-3)$ on the normalization of $S_0$ and to $(1,-1)$ on a normalization of $S_1$.
\item $S_0\cap S_2$ is the curve corresponding to $(1,-3)$ on the normalization of $S_0$ and to $(1,-1)$ on a normalization of $S_2$.
\item $S_0\cap P$ is the curve corresponding to $(0,-1)$ on the normalization of $S_0$ and the cuspidal cubic curve $x^3-yz^2 = 0$ on $P$ with homogeneous coordinates $x,y,z$.
\item $S_1\cap P$ is the curve corresponding to $(1,0)$ on the normalization of $S_1$ and to the line $y = 0$ on $P$ with homogeneous coordinates $x,y,z$ (note that this is the flex tangent of the cuspidal cubic curve $S_0\cap P$).
\item $S_2\cap P$ is the point corresponding to the cone spanned by rays $(0,1),(1,-1)$ on the normalization of $S_2$ and to the point $x=z=0$ on $P$ with homogeneous coordinates $x,y,z$ (note that this is the cusp of the cubic curve $S_0\cap P$).
\end{enumerate}
\end{theorem}

We give an argument in the case (c) using the notation from the proof of Theorem~\ref{thm:central-fibre}.
\begin{proof}[Proof of (c)]
On $Z_0$ the intersection $Z_0\cap Z_P$ is cut out by equations $w_{21}=w_{22}=w_{23} = 0$. Hence it contains as a dense subset the orbit of the toric variety $Z_0$ which corresponds to the one-dimensional cone $\tau = cone(e_2+e_3)$, since $\sigma^{\vee} \setminus \tau^{\perp} \ni v_{21},v_{22},v_{23}$.
Then we have $Q(e_2+e_3) = (0,-1)$.

On $Z_P$ the intersection $Z_0\cap Z_P$ is cut out by the equation $t=0$. By the construction of the isomorphism $P\cong \bP^2$ this equation yields the curve $x^3-yz^2=0$ on $\bP^2$.
\end{proof}

The next lemma shows that all nef line bundles on $X$ are  globally generated, which will be important in the next sections. Note that it is already known that all nef line bundles on $X$ are globally generated by work of~\cite{AW}.

\begin{lemma}\label{lemma:ggbundles}
$L_1+L_2$ and $L_1$ are globally generated line bundles on $X$.
\end{lemma}

\begin{proof}
Since $L_1+L_2$ and $L_1$ are invariant with respect to the $T$-action and the base point locus of a linear system is closed for both linear systems $|L_1+L_2|$ and $|L_1|$ it either has to be empty or it has a nontrivial intersection with the central fibre. The assertion follows by inspecting the weights of the generators of the Cox ring with respect to the Picard torus action and the equations of components of the fibre $p^{-1}([0])$ where $p\colon\Spec \cR \to \bC^4/G$ is as in Proposition~\ref{prop:compontents-over-fibre}. It turns out that the intersections of the zero sets of elements of each of these two linear systems with $p^{-1}([0])$ are contained in the unstable locus.
\end{proof}

\subsection{Smoothness of the GIT quotient}
\label{section:smoothness-of-GIT-quotient}

Let $\cR$ be the subring of the Cox ring of the crepant resolution generated by the elements from the statement of Theorem~\ref{thm:Cox-ring}. In this section we show that the GIT quotient $\Spec \cR/\!/_{L}\bT$ with respect to the linearization of the trivial line bundle by the character $L=2L_1+L_2$ of the Picard torus $\bT$ is smooth. In consequence we see that $\Spec \cR/\!/_L\bT\to \bC^4$ is a crepant resolution. This makes the results on the geometry of crepant resolutions of $\bC^4/G$ in the previous section unconditional and helps to conclude that $\cR$ is the whole Cox ring in the final Section~\ref{section:Cox-LS}.

We consider $\Spec \cR$ as a closed subvariety of $\bC^{20}$ via the embedding given by generators from statement of Theorem~\ref{thm:Cox-ring}.

\begin{theorem}\label{thm:smoothness-of-GIT-quotient}
The stable locus of $\Spec \cR$ with respect to the linearization of the trivial line bundle by a character $(a,b), \ a>b>0$ is covered by seven $T\times \bT$-invariant open subsets $U_{1},\ldots, U_{7}$ such that $U_{i}/\bT\cong \bC^4$. More precisely if $(a,b) = (2,1)$ then:
\begin{enumerate}
\item $U_{1} = \{w_{12}s\neq 0\}$ and $(\cR_{w_{12}s})^{\bT} = \bC\left[w_{02},w_{05},\frac{w_{23}}{w_{12}s},\frac{w_{24}}{w_{12}s}\right]$,
\item $U_{2} = \{w_{12}w_{23} \neq 0\}$ and $(\cR_{w_{12}w_{23}})^{\bT} = \bC\left[\frac{w_{12}^2s}{w_{23}},\frac{w_{21}}{w_{23}},\frac{w_{24}}{w_{23}},\frac{w_3}{w_{12}w_{23}}\right]$,
\item $U_{3} = \{w_{12}w_3\neq 0\}$ and $(\cR_{w_{12}w_3})^{\bT} = \bC\left[\frac{w_{11}}{w_{12}},\frac{w_{13}}{w_{12}},\frac{w_{12}w_{23}}{w_{3}},\frac{w_{12}w_{24}}{w_{3}}\right]$,
\item $U_{4} = \{w_{11}w_3\neq 0\}$ and $(\cR_{w_{11}w_3})^{\bT} = \bC\left[\frac{w_{12}}{w_{11}},\frac{w_{13}}{w_{11}},\frac{w_{11}w_{22}}{w_{3}},\frac{w_{11}w_{23}}{w_{3}}\right]$,
\item $U_{5} = \{w_{13}w_3\neq 0\}$ and $(\cR_{w_{13}w_3})^{\bT} = \bC\left[\frac{w_{11}}{w_{13}},\frac{w_{12}}{w_{13}},\frac{w_{12}w_{21}}{w_{3}},\frac{w_{12}w_{22}}{w_{3}}\right]$,
\item $U_{6} = \{w_{13}w_{22} \neq 0\}$ and $(\cR_{w_{13}w_{22}})^{\bT} = \bC\left[\frac{w_{22}^2t}{w_{13}},\frac{w_{11}}{w_{13}},\frac{w_{14}}{w_{13}},\frac{w_3}{w_{13}w_{22}}\right]$,
\item $U_{7} = \{w_{22}t\neq 0\}$ and $(\cR_{w_{22}t})^{\bT} = \bC\left[w_{03},w_{04},\frac{w_{13}}{w_{22}^2t},\frac{w_{14}}{w_{12}s}\right]$.
\end{enumerate}
In particular the GIT quotient $\Spec \cR/\!/_{(a,b)}\bT$ with respect to the linearization of the trivial line bundle by a character $(a,b), \ a > b > 0$ is smooth. 
\end{theorem}
\begin{proof}
Lemma~\ref{lemma:stability} implies that $\{U_i\}_{i=1,\ldots 7}$ form an open cover of the quotient. It remains to prove equalities from points (1)--(7). Note that then in each case the four generators of the ring on the right-hand side of the equality have to be algebraically independent as the GIT quotient $\Spec \cR/\!/_{(a,b)}\bT$ is irreducible and of dimension four.  

By symmetry it suffices to consider only $U_{i}$ for $i=1,2,3,4$. In each case we calculate the invariants of the localization of the coordinate ring of the ambient $\bC^{20}$, with the help of 4ti2~\cite{4ti2} obtaining in consequence:
\begin{enumerate}
\item $\Spec \cR_{w_{12}s}^{\bT} = \bC\left[w_{0i},\frac{w_{1j}}{w_{12}},\frac{w_{2j}}{w_{12}^2s},\frac{w_{3}}{w_{12}^3s},w_3st\right]_{i = 1,\ldots,7, \ j = 1,\ldots, 5}$
\item $\Spec \cR_{w_{12}w_{23}}^{\bT} = \bC\left[w_{0i},\frac{w_{1j}}{w_{12}},\frac{w_{2j}}{w_{23}},\frac{w_{12}^2s}{w_{23}},\frac{w_{23}^2t}{w_{12}},\frac{w_3}{w_{12}w_{23}}\right]_{i = 1,\ldots,7, \ j = 1,\ldots, 5}$
\item $\Spec \cR_{w_3w_{12}}^{\bT} = \bC\left[w_{0i},\frac{w_{1j}}{w_{12}},\frac{w_{12}w_{2j}}{w_{3}},\frac{w_{12}^3s}{w_{3}},\frac{w_{3}^2t}{w_{12}^3}\right]_{i = 1,\ldots,7, \ j = 1,\ldots, 5}$
\item $\Spec \cR_{w_3w_{11}}^{\bT} = \bC\left[w_{0i},\frac{w_{1j}}{w_{11}},\frac{w_{11}w_{2j}}{w_{3}},\frac{w_{11}^3s}{w_{3}},\frac{w_{3}^2t}{w_{11}^3}\right]_{i = 1,\ldots,7, \ j = 1,\ldots, 5}$
\end{enumerate}

Then, using the Gr\"obner basis of the ideal of relations between generators of $\cR$ with respect to an appropriate lexicographic order, we verify with Singular~\cite{Singular} in each case that each of the generators of these four rings can be expressed as a polynomial of the four generators from the assertion.
\end{proof}

By the inclusion of invariants $\cR^{\bT}\subset \cR$ (see Lemma~\ref{lemma:inclusion-of-invariants}) we have the induced projective map $\Spec \cR/\!/_{L}\bT\to \Spec \cR/\bT \cong \bC^4/G$.

\begin{corollary}\label{corollary-R-gives-resolution}
The map $\varphi\colon \Spec \cR/\!/_L\bT\to \Spec \cR/\bT \cong \bC^4/G$ is a crepant resolution.
\end{corollary}
\begin{proof}
Denote $\ovl{E}_i= \{t_i = 0\}\subset \Spec \cR/\!/_L\bT, \ i=1,2$. These are two irreducible divisors on $\Spec \cR/\!/_L\bT$. By the construction of $\cR$ the map $\varphi$ is an isomorphism outside $\ovl{E}_i$. Hence $\varphi$ is a resolution and it has to be crepant since there are only two crepant divisors over $\bC^4/G$ by symplectic McKay correspondence and they have to be present on each resolution.
\end{proof}

\subsection{Compasses of fixed points}\label{section:compasses}
In this section we obtain a local description of the action of the two-dimensional torus $T$ on a symplectic resolution $X = \Spec \cR /\!/_L \bT$ of $\bC^4/G$ at fixed points of this action, where $L = 2L_1+L_2$. 

We will not use the precise description of the smooth open cover of $X$ from Theorem~\ref{thm:smoothness-of-GIT-quotient} as it was the local calculation that originally led us to this open cover (see Remark~\ref{remark-quotient-coordinates-heuristic} at the end of this section for a detailed explanation). 

Let $x$ be a fixed point of the action of $T$ on $X$. Torus $T$ acts then also on the cotangent space $T_x^* X$. This gives weight-space decomposition $T_x^* X = \bigoplus_{i=1}^{4} V_{\nu_i}$, where $\nu_i$ are in the character lattice of $T$, which we will identify with $\bZ^2$.

The following definition was first introduced in~\cite[Sect~2.3]{BWW}
\begin{definition}\label{def-compass}
The set of weights $\nu_1,\ldots, \nu_4$ is called the \emph{compass} of $x$ in $X$ with respect to the action of $T$.
\end{definition}

We will now work to find compasses of all fixed points of the action $T$ on $X$.
 
\begin{lemma}\label{lemma-polytopes}
The following diagram shows the weights of the action of $T$ on the space of sections of $H^{0}(X,L)$ for $L=2L_1+L_2$ which are nonzero after the restriction to some irreducible component of the central fibre. 

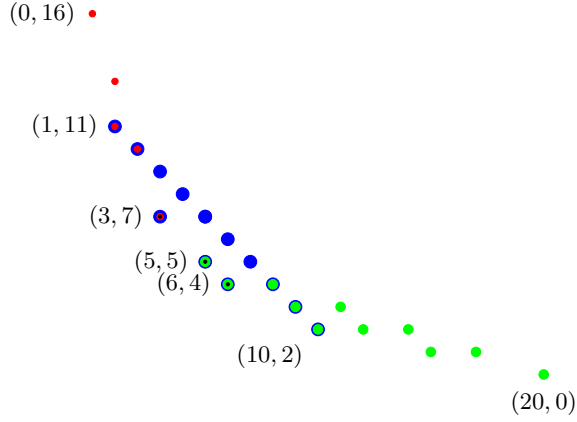
\begin{center}
\begin{tikzpicture}[scale=0.3]
 \foreach \Point in{(8,4), (5,7), (3,9), (9,3), (6,6), (5,5), (2,10), (1,11), (5,7), (4,8), (3,7), (10,2), (7,5), (6,4)}{
       \draw[blue] \Point circle[radius=8pt];
       \fill[blue] \Point circle[radius=8pt];}
       \foreach \Point in {(9,3), (14,2), (5,5), (10,2), (12,2), (15,1), (6,4), (17,1), (8,4), (20,0), (11,3)}{
       \draw[green] \Point circle[radius=6pt];
       \fill[green] \Point circle[radius=6pt];}
       \foreach \Point in {(2,10), (1,11), (1,13), (0,16), (3,7)}{
       \draw[red] \Point circle[radius=4pt];
       \fill[red] \Point circle[radius=4pt];}
       \foreach \Point in {(5,5), (3,7), (6,4)}{
       \draw[black] \Point circle[radius=2pt];
       \fill[black] \Point circle[radius=2pt];}
       
       \coordinate (A) at (0,16);
       \node at (A) [left = 1mm of A]{\small $(0,16)$};
       \coordinate (B) at (1,11);
       \node at (B) [left = 1mm of B]{\small $(1,11)$};
       \coordinate (C) at (3,7);
       \node at (C) [left = 1mm of C]{\small $(3,7)$};
       \coordinate (D) at (5,5);
       \node at (D) [left = 1mm of D]{\small $(5,5)$};
       \coordinate (E) at (6,4);
       \node at (E) [left = 1mm of E]{\small $(6,4)$};
       \coordinate (F) at (10,2);
       \node at (F) [below left = 1mm of F]{\small $(10,2)$};
       \coordinate (G) at (20,0);
       \node at (G) [below = 1mm of G]{\small $(20,0)$};
\end{tikzpicture}
\end{center}
The black dots correspond to the weights of sections of $L$ restricted to $P$, blue ones to $S_0$, green ones to $S_1$ and red ones to $S_2$ (note that lattice points marked by multiple colours correspond to weights occurring in restriction to more than one component).
\end{lemma}
\begin{remark}
Lattice points from Lemma~\ref{lemma-polytopes} are contained in the polyhedron from Remark~\ref{remark:weights-example}. Moreover, their convex hull form the minimal head of this polyhedron. 
\end{remark}
\begin{remark}\label{rmk:fixed-point-weights}
Considering the polytope which is the convex hull of weights marked by dots in Lemma~\ref{lemma-polytopes} we get the polytope of the line bundle $L$ pulled back to $S$ viewed as a toric variety. As $L$ is globally generated (see Lemma~\ref{lemma:ggbundles}) the vertices correspond to the fixed points of the action of $T$ on $S$, cf.~\cite[Lem~2.4(3)]{BWW}. 
In particular one obtains the weights of the action of $T$ on the tangent space to $S$ at fixed points.  
\end{remark}
\begin{theorem}\label{thm-compasses}
The fixed points of the $T$-action correspond to the vertices of the polytopes which are convex hulls of weights marked by the colour fixed in Lemma~\ref{lemma-polytopes}. The compasses of the points corresponding to the vertices of these polytopes are as in the table below:
{\small
\begin{center}
\begin{tabular}{| c | c c c c |}
\hline
Point & & Compass & & \\
\hline
$P_1 \leftrightarrow(0,16)$ & $\nu_{1,1} = (1,-3),$ & $\nu_{1,2} = (1,-5),$ & $\nu _{1,3} = (0,4),$ & $\nu_{1,4} = (0,6)$\\
\hline
$P_2\leftrightarrow(1,11)$ & $\nu_{2,1} = (-1,5),$ & $\nu_{2,2} = (2,-4),$ & $\nu_{2,3} = (1,-1),$ & $\nu_{2,4}=(0,2)$ \\ 
\hline
$P_3 \leftrightarrow(3,7)$ & $\nu_{3,1} = (-2,4),$ & $\nu_{3,2} = (-1,3),$ & $\nu_{3,3} = (2,-2),$ & $\nu_{3,4} = (3,-3)$\\
\hline 
$P_4 \leftrightarrow(5,5)$ & $\nu_{4,1} = (-2,2),$ & $\nu_{4,2} = (1,-1),$ & $\nu_{4,3} = (3,-1),$ & $\nu_{4,4} = (0,2)$\\
\hline 
$P_5 \leftrightarrow(6,4)$ & $\nu_{5,1} = (-1,1),$ & $\nu_{5,2} = (4,-2),$ & $\nu_{5,3} = (-3,3),$ & $\nu_{5,4} = (2,0)$\\
\hline 
$P_6 \leftrightarrow(10,2)$ & $\nu_{6,1} = (-4,2),$ & $\nu_{6,2} = (5,-1),$ & $\nu_{6,3} = (-1,1),$ & $\nu_{6,4} = (2,0)$\\
\hline 
$P_7 \leftrightarrow(20,0)$ & $\nu_{7,1} = (-3,1),$ & $\nu_{7,2} = (-5,1),$ & $\nu_{7,3} = (4,0),$ & $\nu_{7,4} = (6,0)$\\
\hline
\end{tabular}
\end{center}
}
\end{theorem}

The following picture illustrates the weights of the $T$-action calculated in the theorem. It is a directed graph. The points correspond to the sections of $H^{0}(X,L)$ for $L=2L_1+L_2$ which are nonzero after the restriction to the central fibre together with vectors, as in Lemma~\ref{lemma-polytopes}. The directed edges are the vectors from the compasses attached to the points which correspond to fixed points of $T$-action. In case when two vertices are connected by the two edges pointing in both ways we depict them by a single edge without any arrow.

\begin{center}
\begin{tikzpicture}[scale=0.3,>=latex]

\draw[->] (-0.1,16) -- (-0.1,20);
\draw[->] (0.1,16) -- (0.1,22);
\draw[->] (0,16) -- (1,13);
\draw[-] (0,16) -- (1,11);

\draw[->] (1,11) -- (1,13);
\draw[->] (1,11) -- (2,10);
\draw[-] (1,11) -- (3,7);

\draw[->] (3,7) -- (2,10);
\draw[-] (3.05,7.05) -- (5.05,5.05);
\draw[-] (2.85,6.85) -- (5.85,3.85);

\draw[->] (5,5) -- (8,4);
\draw[->] (5,5) -- (5,7);
\draw[-] (5.05,5.05) -- (6.05,4.05);

\draw[-] (6,4) -- (10,2);
\draw[->] (6,4) -- (8,4);

\draw[->] (10,2) -- (9,3);
\draw[->] (10,2) -- (12,2);
\draw[->] (10,2) -- (15,1);

\draw[->] (20,0) -- (15,1);
\draw[->] (20,0) -- (17,1);
\draw[->] (20,0.1) -- (24,0.1);
\draw[->] (20,-0.1) -- (26,-0.1);

 \foreach \Point in{(8,4), (5,7), (3,9), (9,3), (6,6), (5,5), (2,10), (1,11), (5,7), (4,8), (3,7), (10,2), (7,5), (6,4)}{
       \draw[blue] \Point circle[radius=5pt];
       \fill[blue] \Point circle[radius=5pt];}
       \foreach \Point in {(9,3), (14,2), (5,5), (10,2), (12,2), (15,1), (6,4), (17,1), (8,4), (20,0), (11,3)}{
       \draw[green] \Point circle[radius=4pt];
       \fill[green] \Point circle[radius=4pt];}
       \foreach \Point in {(2,10), (1,11), (1,13), (0,16), (3,7)}{
       \draw[red] \Point circle[radius=3pt];
       \fill[red] \Point circle[radius=3pt];}
       \foreach \Point in {(5,5), (3,7), (6,4)}{
       \draw[black] \Point circle[radius=2pt];
       \fill[black] \Point circle[radius=2pt];}
\end{tikzpicture}
\end{center}

In the proof of the theorem we will use two lemmas that are true in the following general setting. Let $G\subset \Sp_4(\bC)$ be a finite group, such that $\bC^4 = V_1\oplus V_2$ as $G$-representations, with $\dim V_i = 2$. Let $\varphi\colon X\to \bC^4/G$ be a symplectic resolution. Then, by~\cite[Thm~1.3]{KaledinSelecta} there is a two-dimensional torus $T$ acting on $X$ and $\bC^4/G$ so that $\varphi$ is equivariant. Clearly $X^T\subset \varphi^{-1}([0])$. Let $x\in X^T$. Then we may decompose $T_x X = T_x(\varphi^{-1}([0]))\oplus V'$ where $V'$ is an eigenspace of the $T$-action. We will call the weights of $T$ on $V'$ the \emph{remaining weights} on $T_x X$. Note that lemmas give an information on the weights of torus actions on $T_x X$ and in the proof of the theorem we will draw conclusions about weights on the dual space $T_x^*X$, by taking negatives of weights on $T_x X$.
\begin{lemma}\label{lemma:trivial-direction}
The remaining weights for the $T$-action on $T_x X$ are of the form $(a,0)$ or $(0,a)$.
\end{lemma}
\begin{proof}
First note that every orbit of the $T$-action on $\bC^4\setminus 0$ is either two-dimensional or has the isotropy group equal to $\bC^*\times 1$ or $1\times \bC^*$.

Now take any $x\in X^T$, any remaining weight $\lambda$ of the $T$-action on $T_x X$ and a one-dimensional eigenspace $V_\lambda$ with this weight which is not contained in the tangent space to the central fibre. By the Luna slice theorem such an eigenspace corresponds to the closure of an orbit $O$ of the $T$-action via an equivariant local \'etale map $U\to T_x X$, where $U$ is an invariant neighbourhood of $x$. In particular $\dim O = 1$, and $\dim \varphi(O) = 1$, as $O$ is not contained in the central fibre of the resolution $\varphi:X\to \bC^4/G$. Therefore $\varphi(O)$ as well as $O$ and $V_{\lambda}$ are stabilized by either $\bC^*\times 1$ or $1\times \bC^*$ and the claim follows.
\end{proof}

For the next claim, consider $\bC^*$ as a subtorus of $T$ embedded with the weight $(1,-1)$.
\begin{lemma}\label{lemma:two-in-two-out}
Let $x\in X^{\bC^*}$. Among the weights of the induced $\bC^*$-action on $T_x X$ two weights are positive and two are negative.
\end{lemma}

\begin{proof}
Let
\begin{gather*}
X_x^+ = \{x\in X \colon \lim_{t\to 0}tx\in X_x\}, \\
X_x^{-} = \{x\in X \colon \lim_{t\to 0}t^{-1}x\in X_x\},
\end{gather*}

where $X_x\subset X^{\bC^*}$ is the connected component containing $x$. We will use the fact that $X_x^{\pm}$ are irreducible, locally closed subsets of $X$ (see~\cite[Thm~4.1]{ABB}).

If at least three weights at $x$ were nonnegative then $\dim X_x^+ \ge 3$ by~\cite[Thm~4.1]{ABB}. Similarly if at least three weights at $x$ are nonpositive then $\dim X_x^- \ge 3$. Suppose that $\dim X_x^+ \ge 3.$ On the other hand $(\bC^4)_0^{+} = \bC^2\times 0$ is two-dimensional and hence also $(\bC^4/G)_0^{+} = (\bC^2\times 0)/G$ is two-dimensional. Since $\dim \varphi^{-1}([0]) = 2$ and the fibres of $\varphi\colon X\setminus \varphi^{-1}([0])\to \bC^4/G$ are of dimension at most one then $\dim \varphi(X_x^{+}) \ge \dim X_{x}^{+}-1 = 2$. As $\varphi(X_x^{+})\subset (\bC^4/G)_0^{+}$ we know that $\dim\varphi(X_x^{+})= 2$ and hence $X_{x}^{+}$ has to be an exceptional divisor of the resolution $\varphi:X\to \bC^4/G$. But the image of such an exceptional divisor is contained in the singular locus of $\bC^{4}/G$ which consists of the image of four planes in $\bC^4$ that have one-dimensional intersection with $\bC^2\times 0$ and so, $\dim \varphi(X_x^{+})= 1$, a contradiction. The case $\dim X_x^- \ge 3$ is completely analogous.
\end{proof}

\begin{proof}[Proof of Theorem~\ref{thm-compasses}]
First, note that $X^T$ is contained in the central fibre. In particular $X^T = S^T$ and as we noted in the Remark~\ref{rmk:fixed-point-weights} the elements of $S^T$ correspond to the vertices of the polytope from the statement.
Taking into account the natural inclusion of the tangent space to the central fibre into the tangent space of $X$ most of the weights can be deduced from the fact that the action of $T$ on $S$ is toric. The polytope spanned by the points marked by dots in the Lemma~\ref{lemma-polytopes} is the polytope of this toric variety. Thus using toric methods one can describe affine cover of $S$ and the weights of the $T$-action on the tangent space to its $T$-fixed points.
Altogether the weights calculated in this way are the ones from the assertion except $\nu_{1,3},\nu_{1,4},\nu_{3,3},\nu_{3,4}$.

Now the calculation of all the weights for the homothety action is easy, since the symplectic form is of weight two with respect to this action, and we can compute at least two weights at each $T$-fixed point, by summing components of each known weight $\nu_{i,j}$. For the remaining weights of the $T$-action we combine Lemmas~\ref{lemma:trivial-direction} and~\ref{lemma:two-in-two-out}. 

For example we know that $\nu_{1,1} = (1,-3)$ and $\nu_{1,2} = (1,-5)$, which gives the weights of the homothety action $-2$ and $-4$ at point corresponding to $(0,16)$. By Proposition~\ref{proposition:weights-on-symplectic-variety} below the remaining weights for homothety are equal to $4$ and $6$. By Lemma~\ref{lemma:trivial-direction} the weight $\nu_{1,3}$ is of the form $(4,0)$ or $(0,4)$ and $\nu_{1,4}$ is of the form $(6,0)$ or $(0,6)$. Since $\nu_{1,1}$ and $\nu_{1,2}$ yield two positive weights for the $\bC^*$-action considered in the Lemma~\ref{lemma:two-in-two-out} and so do $(6,0)$ and $(4,0)$ we have $\nu_{1,3} = (0,4)$ and $\nu_{1,4} = (0,6)$. Other weights are computed analogously.
\end{proof}

\begin{proposition}\label{proposition:weights-on-symplectic-variety}
If $X$ is a smooth $2n$-dimensional variety with a $\bC^*$-action and a symplectic form $\omega$ of weight $k$ then at each fixed point $x$ of $X$ there is a basis $u_1,v_1,\ldots,u_n, v_n$ of $T_x X$ consisting of eigenvectors of $\bC^*$-action, such that the sum of weights of the $\bC^*$-action on $u_i$ and $v_i$ is equal to $-k$ for each $i=1,\ldots, n$.
\end{proposition}

\begin{remark}\label{remark-quotient-coordinates-heuristic}
Note that Theorem~\ref{thm-compasses} follows immediately by Theorem~\ref{thm:smoothness-of-GIT-quotient} when we note that $U_i\cong \bC^4$ is a $T$-invariant neighbourhood of the $i$-th $T$-invariant point (we order points as in rows of the table from the assertion). Nevertheless, since we used the statement of Theorem~\ref{thm-compasses} to guess the open cover $U_i$, we decided to give an independent proof to preserve the logical consequence of our considerations.

More precisely, to find isomorphisms $U_{i}\cong \bC^4$ in Theorem~\ref{thm:smoothness-of-GIT-quotient} we used the compass calculation from this section with a priori assumption on smoothness of the quotient $\Spec \cR/\!/_L\bT$ as a heuristic. To guess the coordinates on each invariant open subset $U_i$ we picked an element $f\in\cR$ of a degree corresponding to the $i$-th fixed point and four elements of $(\cR_f)^{\bT}$ of degrees equal to the predicted weights of the action on the cotangent space.
\end{remark}

\subsection{Dimensions of movable linear systems}\label{section:hilbert}
In this section we use the torus $T$ action on $X$ to give a formula for dimensions of these graded pieces of $\cR(X)$ which correspond to the movable linear systems on some of the resolutions. 

Let $X\to \bC^4/G$ be the resolution corresponding to the linearization of the Picard torus action by a character $(2,1)$. Let $X'\to \bC^4/G$ be the resolution corresponding to the linearization $(1,2)$.

Denote by $P_i$ the fixed points of the $T$-action on $X$ as in the table from Theorem~\ref{thm-compasses}. Let $\{\nu_{i,j}\}_{j=1}^{4}$ denote the compass of $P_i$ in $X$. Let us also denote by $\mu_i(L)$ the weight of the $T$-action on the fibre of $L$ over $P_i$. Note that $\mu_i$ is linear i.e. $\mu_{i}(A+B) = \mu_i(A)+\mu_i(B)$.

\begin{remark}
By Lemma~\ref{lemma:ggbundles} we may compute the weights $\mu_i$ for line bundles $L_1+L_2$ and $L_1$ similarly as for $2L_1+L_2$ in Section~\ref{section:compasses} to obtain (the last column is calculated from the first two ones by the linearity of $\mu_i$):
{\small
\begin{center}
\begin{tabular}{| c | c | c | c |}
\hline
$i$ & $\mu_i(L_1)$ & $\mu_i(L_1+L_2)$ & $\mu_i(L_2)$\\
\hline
$1$ & $(0,4)$ & $(0,12)$ & $(0,8)$\\
\hline
$2$ & $(0,4)$ & $(1,7)$ & $(1,3)$\\
\hline
$3$ & $(0,4)$ & $(3,3)$ & $(3,-1)$\\
\hline
$4$ & $(2,2)$ & $(3,3)$ & $(1,1)$\\
\hline
$5$ & $(3,1)$ & $(3,3)$ & $(0,2)$\\
\hline
$6$ & $(3,1)$ & $(7,1)$ & $(4,0)$\\
\hline
$7$ & $(8,0)$ & $(12,0)$ & $(4,0)$\\
\hline
\end{tabular}
\end{center}
}
\end{remark}

We may now prove the observation on the subdivision of the cone of movable divisors on $X$, see Proposition~\ref{prop-mov-cone} (which is originally due to the work of~\cite{AW}).

\begin{proposition}\label{proposition-mov-subdivision}
There are two symplectic resolutions of $\bC^{4}/G$. The chambers in $\Mov(X)$ corresponding to the nef cones of these resolutions are relative interiors of cones $cone(L_1,L_1+L_2)$ and $cone(L_2,L_1+L_2)$.
\end{proposition}
\begin{proof}
Consider the homomorphisms $\mu_{i}:N^1(X)\to \bR^2$. The two walls of the chamber $\cC$ containing $2L_1+L_2$ are corresponding to the contractions of $X$, in particular they identify some $T$-fixed points of $X$. Hence each wall has to be spanned by an element $v\in \Mov(X)$ satisfying $\mu_{i}(v) = \mu_{j}(v)$ for some $i\neq j$. The only such elements in $cone(L_1,L_1+L_2)$ are lying on the rays spanned by $L_1+L_2$ and $L_1$. Therefore $\cC = cone(L_1,L_1+L_2)$. The analogous argument, using the homomorphisms $\mu_{i}'\colon N^1(X')\to \bR^2$ corresponding to the $T$-fixed points of $X'$, shows that the chamber containing $L_1+2L_2$ is equal to $cone(L_1+L_2,L_2)$.
\end{proof}

\begin{theorem}\label{thm:generating-functions}
If $h^{0}(X,pL_1+qL_2)_{(a,b)}$ is the dimension of the subspace of sections $H^{0}(X,pL_1+qL_2)$ on which $T$ acts with the weight $(a,b)$, then we have the following generating function for such dimensions for line bundles inside the movable cone: 
\begin{multline*}
\sum_{a,b,p, q\ge 0}h^{0}(X,pL_1+qL_2)_{(a,b)}y_1^p y_2^q t_1^a t_2^b =\\
=\sum_{i=1}^{7}\frac{1}{(1-t^{\mu_i(L_1)}y_1)(1-t^{\mu_i(L_2)}y_2)\prod_{j=1}^4(1-t^{\nu_{i,j}})}.
\end{multline*}
\end{theorem}
The computed generating function may be interpreted as the multivariate Hilbert series of a $\bZ_{\ge 0}^2\times \bZ_{\ge 0}^2$-graded subalgebra of $\cR(X)$ consisting of the graded pieces of $\cR(X)$ corresponding to movable line bundles on $X$. This is the interpretation of the theorem that we will use in the next section.

\begin{proof}
By a corollary of the Lefschetz-Riemman-Roch theorem~\cite[Cor~A.3]{BWW}
we have:
\begin{equation*}
\chi^T(X,L) = \sum_{i=1}^{7}\frac{t^{\mu_i(L)}}{\prod_{j=1}^{4}(1 - t^{\nu_{i,j}})}.
\end{equation*}

Using the linearity of $\mu_i$:
\begin{multline*}
\sum_{p, q\ge 0}\chi^T(X,pL_1+qL_2)y_1^p y_2^q =\sum_{p,q\ge 0} \sum_{i=1}^{7}\frac{t^{p\mu_i(L_1)}\cdot t^{q\mu_i(L_2)}}{\prod_{j=1}^{4}(1 - t^{\nu_{i,j}})}y_1^p y_2^q =\\ 
\sum_{i=1}^{7}\frac{1}{(1-t^{\mu_i(L_1)}y_1)(1-t^{\mu_i(L_2)}y_2)\prod_{j=1}^4(1-t^{\nu_{i,j}})}.
\end{multline*}

The assertion follows now by Kawamata-Viehweg vanishing, which implies 
\begin{equation*}
\chi^T(X,pL_1+qL_2) = \sum_{a,b\ge 0}h^{0}(X,pL_1+qL_2)_{(a,b)}t_1^{a}t_2^{b} = \sum_{a,b\ge 0}h^{0}(X',pL_1+qL_2)_{(a,b)}t_1^{a}t_2^{b}.
\end{equation*}
if $p\ge q \ge 0$ and likewise for $q \ge p \ge 0$ on $X'$.
\end{proof}

\begin{example}
The dimensions of the weight spaces corresponding to the lattice points in a head of the polyhedron spanned by weights for the line bundle $2L_1 + L_2$ considered in Remark~\ref{remark:weights-example} and in Section~\ref{section:compasses} can be depicted on the following diagram:
\begin{center}
\begin{tikzpicture}[scale=0.3,>=latex]
\tikzstyle{every node}=[font=\scriptsize]
\node at (0,16) {1};
\node at (0,20) {1};
\node at (0,22) {1};
\node at (1,11) {1};
\node at (1,13) {1};
\node at (1,15) {1};
\node at (1,17) {2};
\node at (1,19) {2};
\node at (1,21) {2};
\node at (2,10) {1};
\node at (2,12) {1};
\node at (2,14) {2};
\node at (2,16) {2};
\node at (2,18) {3};
\node at (2,20) {3};
\node at (2,22) {4};
\node at (3,7) {1};
\node at (3,9) {1};
\node at (3,11) {2};
\node at (3,13) {3};
\node at (3,15) {3};
\node at (3,17) {4};
\node at (3,19) {5};
\node at (3,21) {5};
\node at (4,8) {2};
\node at (4,10) {2};
\node at (4,12) {3};
\node at (4,14) {4};
\node at (4,16) {5};
\node at (4,18) {5};
\node at (4,20) {7};
\node at (4,22) {7};
\node at (5,5) {1};
\node at (5,7) {2};
\node at (5,9) {3};
\node at (5,11) {4};
\node at (5,13) {5};
\node at (5,15) {6};
\node at (5,17) {7};
\node at (5,19) {8};
\node at (5,21) {9};
\node at (6,4) {1};
\node at (6,6) {2};
\node at (6,8) {3};
\node at (6,10) {5};
\node at (6,12) {5};
\node at (6,14) {7};
\node at (6,16) {8};
\node at (6,18) {9};
\node at (6,20) {10};
\node at (6,22) {12};
\node at (7,5) {2};
\node at (7,7) {3};
\node at (7,9) {4};
\node at (7,11) {6};
\node at (7,13) {7};
\node at (7,15) {8};
\node at (7,17) {10};
\node at (7,19) {11};
\node at (7,21) {12};
\node at (8,4) {2};
\node at (8,6) {3};
\node at (8,8) {5};
\node at (8,10) {6};
\node at (8,12) {8};
\node at (8,14) {9};
\node at (8,16) {11};
\node at (8,18) {12};
\node at (8,20) {14};
\node at (8,22) {15};
\node at (9,3) {1};
\node at (9,5) {3};
\node at (9,7) {5};
\node at (9,9) {6};
\node at (9,11) {8};
\node at (9,13) {10};
\node at (9,15) {11};
\node at (9,17) {13};
\node at (9,19) {15};
\node at (9,21) {16};
\node at (10,2) {1};
\node at (10,4) {2};
\node at (10,6) {4};
\node at (10,8) {6};
\node at (10,10) {8};
\node at (10,12) {9};
\node at (10,14) {12};
\node at (10,16) {13};
\node at (10,18) {15};
\node at (10,20) {17};
\node at (10,22) {19};
\node at (11,3) {2};
\node at (11,5) {4};
\node at (11,7) {6};
\node at (11,9) {8};
\node at (11,11) {10};
\node at (11,13) {12};
\node at (11,15) {14};
\node at (11,17) {16};
\node at (11,19) {18};
\node at (11,21) {20};
\node at (12,2) {1};
\node at (12,4) {4};
\node at (12,6) {5};
\node at (12,8) {8};
\node at (12,10) {10};
\node at (12,12) {12};
\node at (12,14) {14};
\node at (12,16) {17};
\node at (12,18) {18};
\node at (12,20) {21};
\node at (12,22) {23};
\node at (13,3) {2};
\node at (13,5) {5};
\node at (13,7) {7};
\node at (13,9) {9};
\node at (13,11) {12};
\node at (13,13) {14};
\node at (13,15) {16};
\node at (13,17) {19};
\node at (13,19) {21};
\node at (13,21) {23};
\node at (14,2) {2};
\node at (14,4) {4};
\node at (14,6) {7};
\node at (14,8) {9};
\node at (14,10) {12};
\node at (14,12) {14};
\node at (14,14) {17};
\node at (14,16) {19};
\node at (14,18) {22};
\node at (14,20) {24};
\node at (14,22) {27};
\node at (15,1) {1};
\node at (15,3) {3};
\node at (15,5) {6};
\node at (15,7) {9};
\node at (15,9) {11};
\node at (15,11) {14};
\node at (15,13) {17};
\node at (15,15) {19};
\node at (15,17) {22};
\node at (15,19) {25};
\node at (15,21) {27};
\node at (16,2) {2};
\node at (16,4) {5};
\node at (16,6) {7};
\node at (16,8) {11};
\node at (16,10) {13};
\node at (16,12) {16};
\node at (16,14) {19};
\node at (16,16) {22};
\node at (16,18) {24};
\node at (16,20) {28};
\node at (16,22) {30};
\node at (17,1) {1};
\node at (17,3) {4};
\node at (17,5) {7};
\node at (17,7) {10};
\node at (17,9) {13};
\node at (17,11) {16};
\node at (17,13) {19};
\node at (17,15) {22};
\node at (17,17) {25};
\node at (17,19) {28};
\node at (17,21) {31};
\node at (18,2) {3};
\node at (18,4) {6};
\node at (18,6) {9};
\node at (18,8) {12};
\node at (18,10) {16};
\node at (18,12) {18};
\node at (18,14) {22};
\node at (18,16) {25};
\node at (18,18) {28};
\node at (18,20) {31};
\node at (18,22) {35};
\node at (19,1) {1};
\node at (19,3) {4};
\node at (19,5) {8};
\node at (19,7) {11};
\node at (19,9) {14};
\node at (19,11) {18};
\node at (19,13) {21};
\node at (19,15) {24};
\node at (19,17) {28};
\node at (19,19) {31};
\node at (19,21) {34};
\node at (20,0) {1};
\node at (20,2) {3};
\node at (20,4) {7};
\node at (20,6) {10};
\node at (20,8) {14};
\node at (20,10) {17};
\node at (20,12) {21};
\node at (20,14) {24};
\node at (20,16) {28};
\node at (20,18) {31};
\node at (20,20) {35};
\node at (20,22) {38};
\node at (21,1) {2};
\node at (21,3) {5};
\node at (21,5) {9};
\node at (21,7) {13};
\node at (21,9) {16};
\node at (21,11) {20};
\node at (21,13) {24};
\node at (21,15) {27};
\node at (21,17) {31};
\node at (21,19) {35};
\node at (21,21) {38};
\node at (22,2) {4};
\node at (22,4) {7};
\node at (22,6) {11};
\node at (22,8) {15};
\node at (22,10) {19};
\node at (22,12) {22};
\node at (22,14) {27};
\node at (22,16) {30};
\node at (22,18) {34};
\node at (22,20) {38};
\node at (22,22) {42};
\node at (23,1) {2};
\node at (23,3) {6};
\node at (23,5) {10};
\node at (23,7) {14};
\node at (23,9) {18};
\node at (23,11) {22};
\node at (23,13) {26};
\node at (23,15) {30};
\node at (23,17) {34};
\node at (23,19) {38};
\node at (23,21) {42};
\node at (24,0) {1};
\node at (24,2) {4};
\node at (24,4) {9};
\node at (24,6) {12};
\node at (24,8) {17};
\node at (24,10) {21};
\node at (24,12) {25};
\node at (24,14) {29};
\node at (24,16) {34};
\node at (24,18) {37};
\node at (24,20) {42};
\node at (24,22) {46};
\node at (25,1) {2};
\node at (25,3) {6};
\node at (25,5) {11};
\node at (25,7) {15};
\node at (25,9) {19};
\node at (25,11) {24};
\node at (25,13) {28};
\node at (25,15) {32};
\node at (25,17) {37};
\node at (25,19) {41};
\node at (25,21) {45};
\node at (26,0) {1};
\node at (26,2) {5};
\node at (26,4) {9};
\node at (26,6) {14};
\node at (26,8) {18};
\node at (26,10) {23};
\node at (26,12) {27};
\node at (26,14) {32};
\node at (26,16) {36};
\node at (26,18) {41};
\node at (26,20) {45};
\node at (26,22) {50};

\draw[very thin] (0,22)--(0,20)--(0,16)--(1,11)--(3,7)--(5,5)--(6,4)--(10,2)--(20,0)--(24,0)--(26,0);

\end{tikzpicture}
\end{center}
\end{example}

\subsection{The Cox ring}\label{section:Cox-LS}
Using the generating function from Theorem~\ref{thm:generating-functions} we may now prove Theorem~\ref{thm:Cox-ring}, i.e. show that the elements from the statement are indeed sufficient to generate the Cox ring of $X$. We outline an argument using the methods of Section~\ref{section:CM-estimation}.

Denote by $\cR(X)$ the Cox ring of $X$ and by $\cR$ the subring of $\cR(X)$ generated by the elements from the statement of Theorem~\ref{thm:Cox-ring}.

\begin{lemma}
The Cox ring of $X$ is generated by $s,t$ and by
\begin{equation*}
\cR(X)_{\ge 0} := \bigoplus_{p,q\ge 0}H^{0}(X, pL_1 + qL_2).
\end{equation*}
\end{lemma}
\begin{proof}
This is a particular case of Proposition~\ref{prop:Mov-and-fixed-generate-Cox}.
\end{proof}

Thus to prove that $\cR = \cR(X)$ it suffices to show that $\cR$ contains $\cR(X)_{\ge 0}$. The following lemma reduces the problem further, to finitely many graded pieces with respect to $\bZ^2$-grading by characters of the Picard torus of $X$.

\begin{lemma}\label{lemma-useCMregularity}
$\cR(X)_{\ge 0}$ is generated by the sections of all linear spaces $H^{0}(X,L)$ for $L \in \cS\cup \cS'$ where $\cS:=\{\cO_X,L_1,L_1+L_2,2L_1,2L_1+L_2,2L_1+2L_2,3L_1+L_2,3L_1+2L_2,4L_1+2L_2\}$ and $\cS':= \{L_2,2L_2,L_1+2L_2,L_1+3L_2,2L_1+3L_2,2L_1+4L_2\}$.
\end{lemma}
\begin{proof}
By virtue of Lemma~\ref{lemma:ggbundles} this follows from Proposition~\ref{prop:CMestimate} with $r=2$ and all $m_1 = m_2 = 1$ for cones $cone(L_1,L_1+L_2)$ and $cone(L_2,L_1+L_2)$..
\end{proof}

Hence we reduced the problem to showing that $\cR$ contains spaces of global sections only for these finitely many line bundles which are elements of $\cS\cup \cS'$ in Lemma~\ref{lemma-useCMregularity}. This, with the help of computer algebra, can be done with the use of the previous section, namely by Theorem~\ref{thm:generating-functions} in which we computed the Hilbert series of $\cR(X)_{\ge 0}$. 

\begin{lemma}
$\cR$ contains $H^0(X,L)$ for each $L\in \cS\cup \cS'$, where $\cS$ is as in the Lemma~\ref{lemma-useCMregularity}.
\end{lemma}
\begin{proof}
We calculate the Hilbert series of $\cR$ in Macaulay2~\cite{M2}. It is of the form:
\begin{equation*}
\frac{1}{(1-y_1^{-2}y_2)(1-y_1y_2^{-2})}\cdot F(y_1,y_2,t_1,t_2),
\end{equation*}
where:
\begin{equation*}
F(y_1,y_2,t_1,t_2) = \frac{f(y_1,y_2,t_1,t_2)}{(1-t_2^4)(1-t_1^4)(1-y_2t_1t_2^3)(1-y_2t_1^4)(1-y_1t_2^4)(1-y_1t_1^3t_2)(1-y_1y_2t_1^3t_2^3)}
\end{equation*}
and 
{\tiny
\begin{align*}
f &= 1+y_1t_1^2t_2^2+y_2t_1^2t_2^2+t_1t_2+y_1^2t_1^4t_2^4+
y_1y_2t_1^4t_2^4+y_2^2t_1^4t_2^4+y_1t_1^5t_2+y_1t_1^3t_2^3+y_1t_1^2t_2^4+y_2t_1^4t_2^2+y_2t_1^3t_2^3+
y_2t_1t_2^5+t_1^3t_2+\\ 
& t_1^2t_2^2+t_1t_2^3-y_1^2y_2t_1^8t_2^4-y_1^2y_2t_1^6t_2^6-y_1^2y_2t_1^5t_2^7-y_1y_2^2t_1^7t_2^5-y_1y_2^2t_1^6t_2^6-y_1y_2^2t_1^4t_2^8-y_1^2t_1^5t_2^5-y_1y_2t_1^7t_2^3-2y_1y_2t_1^6t_2^4-\\
& y_1y_2t_1^5t_2^5-2y_1y_2t_1^4t_2^6-y_1y_2t_1^3t_2^7-y_2^2t_1^5t_2^5-y_1t_1^6t_2^2-y_1t_1^4t_2^4-y_1t_1^3t_2^5-y_1t_1^2t_2^6-y_2t_1^6t_2^2-y_2t_1^5t_2^3-y_2t_1^4t_2^4-y_2t_1^2t_2^6+\\ 
& t_1^6+t_1^4t_2^2+t_1^3t_2^3+t_1^2t_2^4+t_2^6+y_1^2y_2^2t_1^9t_2^7+y_1^2y_2^2t_1^7t_2^9+
y_1^2y_2t_1^6t_2^8+y_1y_2^2t_1^8t_2^6-y_1^2t_1^8t_2^4-y_1^2t_1^4t_2^8-y_1y_2t_1^9t_2^3-y_1y_2t_1^7t_2^5-\\ 
& 2y_1y_2t_1^6t_2^6-y_1y_2t_1^5t_2^7-y_1y_2t_1^3t_2^9-y_2^2t_1^8t_2^4-y_2^2t_1^4t_2^8-y_1t_1^9t_2-y_1t_1^7t_2^3-2y_1t_1^6t_2^4-2y_1t_1^5t_2^5-y_1t_1^4t_2^6-2y_1t_1^3t_2^7-y_1t_1^2t_2^8-\\ 
& y_2t_1^8t_2^2-2y_2t_1^7t_2^3-y_2t_1^6t_2^4-2y_2t_1^5t_2^5-2y_2t_1^4t_2^6-y_2t_1^3t_2^7-y_2t_1t_2^9+y_1^2y_2t_1^{12}t_2^4+y_1^2y_2t_1^{10}t_2^6+2y_1^2y_2t_1^9t_2^7+2y_1^2y_2t_1^8t_2^8+\\
& y_1^2y_2t_1^7t_2^9+2y_1^2y_2t_1^6t_2^{10}+y_1^2y_2t_1^5t_2^{11}+y_1y_2^2t_1^{11}t_2^5+2y_1y_2^2t_1^{10}t_2^6+ y_1y_2^2t_1^9t_2^7+2y_1y_2^2t_1^8t_2^8+2y_1y_2^2t_1^7t_2^9+y_1y_2^2t_1^6t_2^{10}+\\ 
& y_1y_2^2t_1^4t_2^{12}+ y_1^2t_1^9t_2^5+y_1^2t_1^5t_2^9+y_1y_2t_1^{10}t_2^4+y_1y_2t_1^8t_2^6+2y_1y_2t_1^7t_2^7+y_1y_2t_1^6t_2^8+y_1y_2t_1^4t_2^{10}+y_2^2t_1^9t_2^5+ y_2^2t_1^5t_2^9-y_1t_1^5t_2^7-\\
& y_2t_1^7t_2^5-t_1^6t_2^4-t_1^4t_2^6-y_1^2y_2^2t_1^{13}t_2^7-y_1^2y_2^2t_1^{11}t_2^9-y_1^2y_2^2t_1^{10}t_2^{10}-y_1^2y_2^2t_1^9t_2^{11}-y_1^2y_2^2t_1^7t_2^{13}+y_1^2y_2t_1^{11}t_2^7+y_1^2y_2t_1^9t_2^9+\\ 
& y_1^2y_2t_1^8t_2^{10}+y_1^2y_2t_1^7t_2^{11}+y_1y_2^2t_1^{11}t_2^7+y_1y_2^2t_1^{10}t_2^8+y_1y_2^2t_1^9t_2^9+y_1y_2^2t_1^7t_2^{11}+y_1^2t_1^8t_2^8+y_1y_2t_1^{10}t_2^6+ 2y_1y_2t_1^9t_2^7+y_1y_2t_1^8t_2^8+\\ 
& 2y_1y_2t_1^7t_2^9+y_1y_2t_1^6t_2^{10}+y_2^2t_1^8t_2^8+y_1t_1^9t_2^5+y_1t_1^7t_2^7+
y_1t_1^6t_2^8+y_2t_1^8t_2^6+y_2t_1^7t_2^7+y_2t_1^5t_2^9-y_1^2y_2^2t_1^{12}t_2^{10}-y_1^2y_2^2t_1^{11}t_2^{11}-\\
& y_1^2y_2^2t_1^{10}t_2^{12}-y_1^2y_2t_1^{12}t_2^8-y_1^2y_2t_1^{10}t_2^{10}-y_1^2y_2t_1^9t_2^{11}-y_1y_2^2t_1^{11}t_2^9-y_1y_2^2t_1^{10}t_2^{10}-y_1y_2^2t_1^8t_2^{12}-y_1^2t_1^9t_2^9+y_1y_2t_1^9t_2^9-y_2^2t_1^9t_2^9-\\
& y_1^2y_2^2t_1^{12}t_2^{12}-y_1^2y_2t_1^{11}t_2^{11}-y_1y_2^2t_1^{11}t_2^{11}-y_1^2y_2^2t_1^{13}t_2^{13}. 
\end{align*}
}
 Then, using Singular~\cite{Singular} we extract from it Hilbert series for each of the vector spaces $\cR_L$, $L\in \cS$, graded by characters of $T$. Denote the Hilbert series corresponding to $pL_1+qL_2$ by $F_{p,q}$. Then we have:
 \begin{align*}
  F_{0,0}(t_1,t_2) &= \tfrac{F(1,1,t_1,t_2) + F(\epsilon,\epsilon^2,t_1,t_2) + F(\epsilon^2,\epsilon,t_1,t_2)}{3}, \\
 F_{1,0}(t_1,t_2) &= \tfrac{F(1,1,t_1,t_2) + \epsilon^2 F(\epsilon,\epsilon^2,t_1,t_2) + \epsilon F(\epsilon^2,\epsilon,t_1,t_2)}{3}, \\
 F_{0,1}(t_1,t_2) &= \tfrac{F(1,1,t_1,t_2) + \epsilon F(\epsilon,\epsilon^2,t_1,t_2) + \epsilon^2 F(\epsilon^2,\epsilon,t_1,t_2)}{3},\\
 F_{1,1}(t_1,t_2) &=F_{0,0}(t_1,t_2) - F(0,0,t_1,t_2), \\
 F_{2,0}(t_1,t_2) &= F_{0,1}(t_1,t_2) - \tfrac{\partial F}{\partial y_2}(0,0,t_1,t_2),\\
 F_{2,1}(t_1,t_2) &= F_{1,0}(t_1,t_2) - \tfrac{\partial F}{\partial y_1}(0,0,t_1,t_2) - \tfrac{1}{2}\tfrac{\partial^2 F}{\partial y_2^2}(0,0,t_1,t_2),
 \end{align*}
 \begin{align*}
 F_{2,2}(t_1,t_2) &= F_{0,0}(t_1,t_2) - \tfrac{\partial^{2} F}{\partial y_1\partial y_2}(0,0,t_1,t_2) - \tfrac{1}{3!}\tfrac{\partial^{3} F}{\partial y_1^3}(0,0,t_1,t_2) -\tfrac{1}{3!}\tfrac{\partial^{3} F}{\partial y_2^3}(0,0,t_1,t_2),\\ 
 F_{3,1}(t_1,t_2) &= F_{2,0}(t_1,t_2) - \tfrac{1}{2}\tfrac{\partial^3 F}{\partial y_1\partial y_2^2}(0,0,t_1,t_2)-\tfrac{1}{2}\tfrac{\partial^2 F}{\partial y_1^2}(0,0,t_1,t_2)-\tfrac{1}{4!}\tfrac{\partial^4 F}{\partial y_2^4}(0,0,t_1,t_2),\\
 F_{3,2}(t_1,t_2) &= F_{2,1}(t_1,t_2) - \tfrac{1}{2}\tfrac{\partial^3 F}{\partial y_1^2\partial y_2}(0,0,t_1,t_2) - \tfrac{1}{3!}\tfrac{\partial^4 F}{\partial y_1\partial y_2^3}(0,0,t_1,t_2)- \tfrac{1}{4!}\tfrac{\partial^4 F}{\partial y_1^4}(0,0,t_1,t_2)-\\
 &-\tfrac{1}{5!}\tfrac{\partial^5 F}{\partial y_2^5}(0,0,t_1,t_2),\\
 F_{4,2}(t_1,t_2) &= F_{3,1}(t_1,t_2) - \tfrac{1}{3!}\tfrac{\partial^4 F}{\partial y_1^3\partial y_2}(0,0,t_1,t_2) - \tfrac{1}{2!3!}\tfrac{\partial^5 F}{\partial y_1^2\partial y_2^3}(0,0,t_1,t_2)- \tfrac{1}{5!}\tfrac{\partial^6 F}{\partial y_1\partial y_2^5}(0,0,t_1,t_2)- \\
 & - \tfrac{1}{5!}\tfrac{\partial^5 F}{\partial y_1^5}(0,0,t_1,t_2)-\tfrac{1}{7!}\tfrac{\partial^7 F}{\partial y_2^7}(0,0,t_1,t_2).
 \end{align*}

Using Singular we check that for each $L\in \cS$ the Hilbert series of $\cR_L$ agrees with the Hilbert series for $\cR(X)_L$ which we calculated in Theorem~\ref{thm:generating-functions}. Since $\cR\subset \cR(X)$ it means that $\cR_L = \cR(X)_L$.

As in the case $G = \bZ_2\wr S_2$ by symmetry of the Hilbert series of $\cR$ and of Hilbert series of $\cR(X)_{\ge 0}$ (see Theorem~\ref{thm:generating-functions}) if $H^0(X,aL_1 + bL_2)\subset \cR$ then $H^0(X,bL_1+aL_2)\subset \cR$. In particular from what we proven it follows also that $\cR_L = \cR(X)_L$ for $L\in \cS'$.
\end{proof}

\section{{$S_3$} and {$\bZ_2\wr S_2$}}
\label{section:two-examples}

As outlined in Section~\ref{section:sympl-strategy}, the same method may be applied for other quotients of finite group actions via reducible symplectic representations. Here we present the main result -- the description of the Cox ring -- and list the most important intermediate results in the two simpler examples.
\begin{enumerate}[label = (\alph*)]
\item The action of the symmetric group $S_3$ on $\bC^4$ generated by matrices:
\begin{equation*}
\left(\begin{smallmatrix}
\epsilon & 0 & 0 & 0\\
0 & \epsilon^{-1} & 0 & 0\\
0 & 0 & \epsilon & 0\\
0 & 0 & 0 & \epsilon^{-1}
\end{smallmatrix}\right), \qquad \left(\begin{smallmatrix}
0 & 1 & 0 & 0\\
1 & 0 & 0 & 0\\
0 & 0 & 0 & 1\\
0 & 0 & 1 & 0
\end{smallmatrix}\right),
\end{equation*}
where $\epsilon = e^{2\pi i/3}$ is a $3$rd root of unity.
\item The action of the wreath product $\bZ_2\wr S_2$ on $\bC^4$ by matrices:
\begin{equation*}
\left(\begin{smallmatrix}
1 & 0 & 0 & 0\\
0 & -1 & 0 & 0\\
0 & 0 & 1 & 0\\
0 & 0 & 0 & -1
\end{smallmatrix}\right),
\qquad
\left(\begin{smallmatrix}
0 & 1 & 0 & 0\\
1 & 0 & 0 & 0\\
0 & 0 & 0 & 1\\
0 & 0 & 1 & 0
\end{smallmatrix}\right).
\end{equation*}
\end{enumerate}
These two actions were analysed with use of other methods in~\cite{SymplCox}.

By the McKay correspondence~\cite{ItoReid} the exceptional divisor $E$ on a crepant resolution $X\to \bC^4/G$ is irreducible if $G = S_3$ and there are two irreducible components $E_1,E_2$ if $G = \bZ_2\wr S_2$. As before, using reasoning as in~\cite[Sect~2D]{81resolutions} we can show that free abelian group $\Pic(X)$ has a basis $L$, or, respectively $L_1,L_2$ which are dual (with respect to the intersection form) to the basis of space of numerical classes of curves on $X$ given by generic fibres of restriction of resolutions to the components $E$ and $E_1,E_2$.

Denote:

{\small 
\begin{tabular}{l}
$\phi_{01}= x_{1}y_{1},\quad \phi_{02} = x_{2}y_{2},\quad \phi_{03} = x_{1}y_{2} + x_{2}y_{1},$ \\
$\phi_{04} = x_{1}^{3} + y_{1}^{3}, \ \phi_{05} = x_{2}^{3} + y_{2}^{3},\ \phi_{06} =x_{1}^{2}x_{2} + y_{1}^{2}y_{2}, \ \phi_{07} = x_{1}x_{2}^{2} + y_{1}y_{2}^{2},$\\
$\phi_{11} = x_{1}y_{2} - x_{2}y_{1},$ \\
$\phi_{12} = x_{1}^{3} - y_{1}^{3}, \ \phi_{13} = x_{2}^{3} - y_{2}^{3},\ \phi_{14} = x_{1}^{2}x_{2} - y_{1}^{2}y_{2}, \ \phi_{15} = x_{1}x_{2}^{2} - y_{1}y_{2}^{2},$ \\
\\
$\psi_{01} = x_1^2+y_1^2, \quad \psi_{02} = x_2^2+y_2^2, \quad \psi_{03} = x_1x_2+y_1y_2,$\\
$\psi_{11} = x_1y_2+x_2y_1, \quad \psi_{12} = x_1y_1, \quad \psi_{13} = x_2y_2,$\\
$\psi_{21} = x_1x_2-y_1y_2, \quad \psi_{22} = x_1^2-y_1^2, \quad \psi_{23} = x_2^2-y_2^2,$\\
$\psi_{3} = x_{1}y_{2} - x_{2}y_{1}$.
\end{tabular}
}

\begin{theorem} \label{thm-cox-rings-two-examples}
\leavevmode
\begin{enumerate}[label=(\alph*)]
\item Let $G = S_3$. The Cox ring $\cR(X)\subset \cR(\bC^4/G)[t^{\pm 1}]$ of the (unique) symplectic resolution $X\to \bC^4/G$ is generated by the elements $w_{01},\ldots,w_{07},$ $w_{11},\ldots,w_{15},t$, where $w_{0i} = \phi_{0i}, \ w_{1j} = \phi_{1j}t$. In particular the degree matrix with respect to the generator $L$ of $\Pic(X)$ (the first row) and to the $T$-action (remaining two rows) is:
\setcounter{MaxMatrixCols}{20}
\begin{equation*}
\left(\begin{smallmatrix}
w_{01} & w_{02} & w_{03} & w_{04} & w_{05} & w_{06} & w_{07} & w_{11} & w_{12} & w_{13} & w_{14} & w_{15} & t \\
\\
\hline\\
\\
0 & 0 & 0 & 0 & 0 & 0 & 0 & 1 & 1 & 1 & 1 & 1 & -2\\
\\
\hline\\
\\
2 & 0 & 1 & 3 & 0 & 2 & 1 & 1 & 3 & 0 & 2 & 1 & 0\\
0 & 2 & 1 & 0 & 3 & 1 & 2 & 1 & 0 & 3 & 1 & 2 & 0 \\
\end{smallmatrix}\right).
\end{equation*}

\item Let $G = \bZ_2 \wr S_2$. The Cox ring $\cR(X)\subset \cR(\bC^4/G)[t_1^{\pm 1},t_2^{\pm 1}]$ of the symplectic resolution $X\to \bC^4/G$ is generated by the elements $w_{01},w_{02},w_{03},$ $w_{11},w_{12},w_{13},w_{21},w_{22},w_{23},w_3,$ $s,t$, where $w_{0i} = \psi_{0i}, \ w_{jk} = \psi_{jk}t_j, \ s = t_1^{-2},\ t = t_2^{-2}$. In particular the degree matrix with respect to the generators $L_1, L_2$ of $\Pic(X)$ (the first two rows) and to the $T$-action (remaining two rows) is:
\setcounter{MaxMatrixCols}{20}
\begin{equation*}
\left(\begin{smallmatrix}
w_{01} & w_{02} & w_{03} & w_{11} & w_{12} & w_{13} & w_{21} & w_{22} & w_{23} & w_3 & s & t \\
\\
\hline\\
\\
0 & 0 & 0 & 1 & 1 & 1 & 0 & 0 & 0 & 1 & -2 & 0\\
0 & 0 & 0 & 0 & 0 & 0 & 1 & 1 & 1 & 1 & 0 & -2\\
\\
\hline\\
\\
2 & 0 & 1 & 1 & 2 & 0 & 1 & 2 & 0 & 1 & 0 & 0\\
0 & 2 & 1 & 1 & 0 & 2 & 1 & 0 & 2 & 1 & 0 & 0\\
\end{smallmatrix}\right).
\end{equation*}
\end{enumerate}

\end{theorem}

\begin{theorem}\label{thm:S3}
Let $G = S_3$. Let $\varphi\colon X\to \bC^4/G$ be a crepant resolution.
\begin{enumerate}[label = (\roman*)]
\item The central fibre of $\varphi$ is an irreducible toric surface isomorphic to the projective cone over the image of the third Veronese embedding $\bP^1\to \bP^3$. 
\item The linear system $|L|$ is globally generated.
\item $\varphi$ can be constructed as the map induced between GIT quotients of $\Spec \cR$ with respect to the character of the Picard torus corresponding to $L$ and to the trivial one, where $\cR$ is the ring from the assertion of Theorem~\ref{thm-cox-rings-two-examples}(a).
\item If $h^{0}(X,pL)_{(a,b)}$ is the dimension of the subspace of sections $H^{0}(X,pL)$ on which $T$ acts with the weight $(a,b)$, then we have the following generating function for such dimensions for line bundles inside the movable cone: 
\begin{equation*}
\sum_{a,b,p\ge 0}h^{0}(X,pL)_{(a,b)}y^p t_1^a t_2^b =\sum_{i=1}^{3}\frac{1}{(1-t^{\mu_i}y)\prod_{j=1}^4(1-t^{\nu_{i,j}})},
\end{equation*}
where $\mu_1 = (0,3), \ \mu_2 = (1,1), \ \mu_3 = (3,0)$, and 

{\small
\begin{tabular}{ c c c c }
$\nu_{1,1} = (1,-2),$ & $\nu_{1,2} = (1,-1),$ & $\nu _{1,3} = (0,3),$ & $\nu_{1,4} = (0,2),$\\
$\nu_{2,1} = (-1,2),$ & $\nu_{2,2} = (2,-1),$ & $\nu_{2,3} = (0,1),$ & $\nu_{2,4}=(1,0),$ \\ 
$\nu_{3,1} = (-2,1),$ & $\nu_{3,2} = (-1,1),$ & $\nu_{3,3} = (2,0),$ & $\nu_{3,4} = (3,0).$
\end{tabular}
}
\end{enumerate}
\end{theorem}

In part (c) of the theorem one can give an explicit open cover by affine spaces, as before. The compasses at the fixed points which are used to formulate part (d) are illustrated by the picture below.

\begin{center}
\begin{tikzpicture}[scale=0.5,>=latex]

\draw[->] (-0.1,3) -- (-0.1,6);
\draw[->] (0.1,3) -- (0.1,5);
\draw[->] (0,3) -- (1,2);
\draw[-] (0,3) -- (1,1);

\draw[->] (1,1) -- (1,2);
\draw[->] (1,1) -- (2,1);
\draw[-] (1,1) -- (3,0);

\draw[->] (3,0) -- (2,1);
\draw[->] (3,0.1) -- (5,0.1);
\draw[->] (3,-0.1) -- (6,-0.1);

 \foreach \Point in{(3,0), (2,1), (1,1), (1,2), (0,3)}{
       \draw[black] \Point circle[radius=3pt];
       \fill[black] \Point circle[radius=3pt];}
       
\coordinate (A) at (0,3);
       \node at (A) [left = 1mm of A]{\small $(0,3)$};
       \coordinate (B) at (1,1);
       \node at (B) [left = 1mm of B]{\small $(1,1)$};
       \coordinate (C) at (3,0);
       \node at (C) [left = 1mm of C]{\small $(3,0)$};
\end{tikzpicture}
\end{center}

\begin{theorem}\label{thm:D8}
Let $G = \bZ_2\wr S_2$. Let $\varphi\colon X\to \bC^4/G$ be a crepant resolution.
\begin{enumerate}[label = (\roman*)]
\item The central fibre of $\varphi$ consists of two irreducible components, one of which is isomorphic to the Hirzebruch surface $\cH_4$ and the other one isomorphic to $\bP^2$.
\item The linear systems $|L_1|$ and $|L_1+L_2|$ are globally generated.
\item $\varphi$ can be constructed as the map induced between GIT quotients of $\Spec \cR$ with respect to the character of the Picard torus corresponding to either $2L_1 + L_2$ or $L_1 + 2L_2$ and to the trivial one, where $\cR$ is the ring from the assertion of Theorem~\ref{thm-cox-rings-two-examples}(b).
\item If $h^{0}(X,pL_1+qL_2)_{(a,b)}$ is the dimension of the subspace of sections $H^{0}(X,pL_1+qL_2)$ on which $T$ acts with the weight $(a,b)$, then we have the following generating function for such dimensions for line bundles inside the movable cone: 
\begin{multline*}
\sum_{a,b,p, q\ge 0}h^{0}(X,pL_1+qL_2)_{(a,b)}y_1^p y_2^q t_1^a t_2^b =\\
=\sum_{i=1}^{5}\frac{1}{(1-t^{\mu_i(L_1)}y_1)(1-t^{\mu_i(L_2)}y_2)\prod_{j=1}^4(1-t^{\nu_{i,j}})},
\end{multline*}
where 

{\small
\begin{tabular}{ c c c c }
$\nu_{1,1} = (1,-3),$ & $\nu_{1,2} = (1,-1),$ & $\nu _{1,3} = (0,4),$ & $\nu_{1,4} = (0,2),$\\
$\nu_{2,1} = (-1,3),$ & $\nu_{2,2} = (1,-1),$ & $\nu_{2,3} = (2,-2),$ & $\nu_{2,4}=(0,2),$ \\ 
$\nu_{3,1} = (-1,1),$ & $\nu_{3,2} = (1,-1),$ & $\nu_{3,3} = (2,0),$ & $\nu_{3,4} = (0,2),$\\
$\nu_{4,1} = (3,-1),$ & $\nu_{4,2} = (-1,1),$ & $\nu_{4,3} = (-2,2),$ & $\nu_{4,4} = (2,0),$\\
$\nu_{5,1} = (-3,1),$ & $\nu_{5,2} = (-1,1),$ & $\nu_{5,3} = (4,0),$ & $\nu_{5,4} = (2,0)$\\
\end{tabular}
}

{\small
and $\qquad$
\begin{tabular}{| c | c | c | c | c | c |}
\hline
$i$ & $1$ & $2$ & $3$ & $4$ & $5$ \\
\hline
$\mu_i(L_1)$ & $(0,2)$ & $(0,2)$ & $(1,1)$ & $(2,0)$ & $(2,0)$ \\
\hline 
$\mu_i(L_2)$ & $(0,2)$ & $(0,-1)$ & $(0,0)$ & $(-1,1)$ & $(2,0)$\\
\hline
\end{tabular}
}
\end{enumerate}

\end{theorem}

Again in part (c) of one can give an explicit open cover by affine spaces and the compasses used in (d) may be illustrated by a picture.

\begin{center}
\begin{tikzpicture}[scale=0.5,>=latex]

\draw[->] (-0.1,6) -- (-0.1,8);
\draw[->] (0.1,6) -- (0.1,10);
\draw[->] (0,6) -- (1,5);
\draw[-] (0,6) -- (1,3);

\draw[->] (1,3) -- (1,5);
\draw[-] (1.05,3.05) -- (2.05,2.05);
\draw[-] (0.9,2.9) -- (2.9,0.9);

\draw[->] (2,2) -- (4,2);
\draw[->] (2,2) -- (2,4);
\draw[-] (2.05,2.05) -- (3.05,1.05);

\draw[-] (3,1) -- (6,0);
\draw[->] (3,1) -- (5,1);

\draw[->] (6,-0.1) -- (8,-0.1);
\draw[->] (6,0.1) -- (10,0.1);
\draw[->] (6,0) -- (5,1);

\foreach \Point in{(5,1), (3,3), (1,5), (2,2), (6,0), (4,2), (2,4), (3,1), (0,6), (1,3)}{
      \draw[red] \Point circle[radius=4pt];
      \fill[red] \Point circle[radius=4pt];}
      \foreach \Point in {(2,2), (3,1), (1,3)}{
      \draw[black] \Point circle[radius=2pt];
      \fill[black] \Point circle[radius=2pt];}
      
       \coordinate (A) at (0,6);
       \node at (A) [left = 1mm of A]{\small $(0,6)$};
       \coordinate (B) at (1,3);
       \node at (B) [left = 1mm of B]{\small $(1,3)$};
       \coordinate (C) at (2,2);
       \node at (C) [left = 1mm of C]{\small $(2,2)$};
       \coordinate (D) at (3,1);
       \node at (D) [left = 1mm of D]{\small $(3,1)$};
       \coordinate (E) at (6,0);
       \node at (E) [left = 1mm of E]{\small $(6,0)$};
\end{tikzpicture}
\end{center}

\bibliographystyle{plain}
\bibliography{LS}

\begin{thebibliography}{10}

\bibitem{4ti2}
4ti2 team.
\newblock 4ti2---a software package for algebraic, geometric and combinatorial
  problems on linear spaces.
\newblock www.4ti2.de.

\bibitem{Tvarieties}
Klaus {Altmann}, Nathan~Owen {Ilten}, Lars {Petersen}, Hendrik {S{\"u}{\ss}},
  and Robert {Vollmert}.
\newblock {The geometry of T-Varieties}.
\newblock In Piotr Pragacz, editor, {\em Contributions to Algebraic Geometry --
  a tribute to Oscar Zariski}, EMS Series of Congress Reports, Impanga Lecture
  Notes, pages 17--69, 2012.

\bibitem{AW}
Marco Andreatta and Jaros\l{}aw Wi{\'s}niewski.
\newblock 4-dimensional symplectic contractions.
\newblock {\em Geometriae Dedicata}, 168(1):311--337, 2014.

\bibitem{CoxRings}
Ivan Arzhantsev, Ulrich Derenthal, J\"urgen Hausen, and Antonio Laface.
\newblock {\em Cox {R}ings}.
\newblock Cambridge University Press, New York, 2014.

\bibitem{BFQ}
Paul Baum, William Fulton, and George Quart.
\newblock Lefschetz-{R}iemann-{R}och for singular varieties.
\newblock {\em Acta Math.}, 143:193--211, 1979.

\bibitem{BellamySchedler}
Gwyn Bellamy and Travis Schedler.
\newblock A new linear quotient of {${\bf C}^4$} admitting a symplectic
  resolution.
\newblock {\em Math. Z.}, 273(3-4):753--769, 2013.

\bibitem{ABB}
A.~Bialynicki-Birula.
\newblock Some theorems on actions of algebraic groups.
\newblock {\em Annals of Mathematics}, 98(3):480--497, 1973.

\bibitem{BCHM}
Caucher {Birkar}, Paolo {Cascini}, Christopher~D. {Hacon}, and James
  {McKernan}.
\newblock Existence of minimal models for varieties of log general type.
\newblock {\em J. Amer. Math. Soc.}, 23:405--468, 2010.

\bibitem{BWW}
J.~Buczy\'nski, A.~Weber, and J.~Wi\'sniewski.
\newblock Algebraic torus actions on contact manifolds.
\newblock {\em ArXiv preprint, arXiv:1802.05002}, 2018.

\bibitem{CLS}
David~A. Cox, John~B. Little, and Hal~K. Schenck.
\newblock {\em Toric Varieties}.
\newblock Graduate studies in mathematics. American Mathematical Soc., 2011.

\bibitem{Singular}
Wolfram Decker, Gert-Martin Greuel, Gerhard Pfister, and Hans Sch\"onemann.
\newblock {\sc Singular} {4-0-3} --- {A} computer algebra system for polynomial
  computations.
\newblock \url{http://www.singular.uni-kl.de}, 2016.

\bibitem{CoxSurf}
Maria Donten-Bury.
\newblock {Cox rings of minimal resolutions of surface quotient singularities}.
\newblock {\em Glasg. Math. J.}, 58(1):325--355, 2016.

\bibitem{SymplCox}
Maria Donten-Bury and Maksymilian Grab.
\newblock {Cox rings of some symplectic resolutions of quotient singularities}.
\newblock {\em arXiv:1504.07463}, 2015.

\bibitem{3dimCox}
Maria Donten-Bury and Maksymilian Grab.
\newblock Crepant resolutions of 3-dimensional quotient singularities via cox
  rings.
\newblock {\em Experimental Mathematics}, 0(0):1--20, 2017.

\bibitem{81resolutions}
Maria Donten-Bury and Jaros\l{}aw~A. Wi{\'s}niewski.
\newblock On $81$ symplectic resolutions of a $4$ -dimensional quotient by a
  group of order $32$.
\newblock {\em Kyoto J. Math.}, 57(2):395--434, 06 2017.

\bibitem{GAP4}
The GAP~Group.
\newblock {\em {GAP -- Groups, Algorithms, and Programming, Version 4.7.5}},
  2014.

\bibitem{M2}
Daniel Grayson and Michael Stillman.
\newblock Macaulay2, a software system for research in algebraic geometry.
\newblock Available at http://www.math.uiuc.edu/Macaulay2/, 2013.

\bibitem{multigraded_regularity}
Milena Hering, Hal Schenck, and Gregory~G. Smith.
\newblock Syzygies, multigraded regularity and toric varieties.
\newblock {\em Compositio Mathematica}, 142(6):1499–1506, 2006.

\bibitem{ItoReid}
Yukari Ito and Miles Reid.
\newblock The {M}c{K}ay correspondence for finite subgroups of {${\rm
  SL}(3,\bold C)$}.
\newblock In {\em Higher-dimensional complex varieties ({T}rento, 1994)}, pages
  221--240. de Gruyter, Berlin, 1996.

\bibitem{KaledinSelecta}
D.~Kaledin.
\newblock On crepant resolutions of symplectic quotient singularities.
\newblock {\em Selecta Math. (N.S.)}, 9(4):529--555, 2003.

\bibitem{KaledinMcKay}
Dmitry {Kaledin}.
\newblock Mc{K}ay correspondence for symplectic quotient singularities.
\newblock {\em Invent. Math.}, 148(1):151--175, 2002.

\bibitem{KMM87}
Yujiro Kawamata, Katsumi Matsuda, and Kenji Matsuki.
\newblock Introduction to the minimal model problem.
\newblock In {\em Algebraic Geometry, Sendai, 1985}, pages 283--360, Tokyo,
  Japan, 1987. Mathematical Society of Japan.

\bibitem{LehnSorger}
Manfred {Lehn} and Christoph {Sorger}.
\newblock {A symplectic resolution for the binary tetrahedral group}.
\newblock {\em Michel Brion (ed.), Geometric methods in representation theory.
  II. Selected papers based on the presentations at the summer school,
  Grenoble, France, June 16 – July 4, 2008.}, 24(2):429--435, 2012.

\bibitem{MaclaganSmith}
{D}iane {M}aclagan and {G}regory~{G.} {S}mith.
\newblock Multigraded {C}astelnuovo-{M}umford regularity.
\newblock {\em Journal für die reine und angewandte Mathematik (Crelles
  Journal)}, 571:179--212, 2004.

\bibitem{NamikawaFlops}
Yoshinori Namikawa.
\newblock Flops and {P}oisson deformations of symplectic varieties.
\newblock {\em Publ. RIMS, Kyoto Univ}, 44(1):259--314, 2008.

\bibitem{ReidMcKayCorrespondence}
M.~{Reid}.
\newblock {McKay correspondence}.
\newblock In T.~Katsura, editor, {\em Algebraic Geometry Symposium, Kinosaki,
  November 1996}, pages 14--41, February 1997.

\bibitem{WierzbaWisniewski}
Jan Wierzba and Jaros{\l}aw~A. Wi{\'s}niewski.
\newblock Small contractions of symplectic 4-folds.
\newblock {\em Duke Math. J.}, 120(1):65--95, 2003.

\bibitem{Yamagishi}
Ryo Yamagishi.
\newblock On smoothness of minimal models of quotient singularities by finite
  subgroups of ${SL}_n(\mathbb{C})$.
\newblock {\em Glasgow Mathematical Journal}, 2018.

\end{thebibliography}

\end{document}